\documentclass[12pt,english]{amsart}

\input xy
\xyoption{all}

\usepackage[latin1]{inputenc}
\usepackage{amsfonts,amsmath,amssymb,amsthm}

\newcommand{\Cc}{\mathbb{C}} 
\newcommand{\Pp}{\mathbb{P}}

\newcommand{\Nn}{\mathbb{N}}
\newcommand{\Zz}{\mathbb{Z}}
\newcommand{\Qq}{\mathbb{Q}} 
\newcommand{\Ff}{\mathbb{F}}

{\theoremstyle{plain}
\newtheorem{theorem}{Theorem}[section]    

\newtheorem{twisting lemma}[theorem]{Twisting lemma}
\newtheorem{hilbert-malle}[theorem]{Hilbert-Malle theorem}

\newtheorem{lemma}[theorem]{Lemma}       
\newtheorem{proposition}[theorem]{Proposition}  
\newtheorem{corollary}[theorem]{Corollary}   

\newtheorem{criterion}[theorem]{Criterion}       

}
{\theoremstyle{remark}
\newtheorem{definition}[theorem]{Definition}      
\newtheorem{remark}[theorem]{Remark}   

}

\font\fifteenpc=cmcsc10 scaled 1400

\font\tensf=cmss10
\font\fivesf=cmss10 scaled 500
\font\sevensf=cmss10 scaled 700
\newfam\sffam  \scriptscriptfont\sffam=\fivesf
\textfont\sffam=\tensf
\scriptfont\sffam=\sevensf
\def\sf{\fam\sffam\tensf}

\def\specz{\prec}

\def\verydiff{\hbox{$\hskip 2pt \#\hskip 2pt$}}

\def\Gabs{\hbox{\rm G}}

\def\Gal{\hbox{\rm Gal}}
\def\max{\hbox{\rm max}}

\def\cm{\hbox{\hbox{\rm C}\kern-5pt{\raise 1pt\hbox{$|$}}}}

\def\lhfl#1#2{\smash{\mathop{\hbox to 12mm{\leftarrowfill}}
\limits^{#1}_{#2}}}

\def\rhfl#1#2{\smash{\mathop{\hbox to 12mm{\rightarrowfill}}
\limits^{#1}_{#2}}}

\def\build#1_#2^#3{\mathrel{
\mathop{\kern 0pt#1}\limits_{#2}^{#3}}}

\def\htrait#1#2{\smash{\mathop{\hbox to 12mm{\hrulefill}}
\limits^{#1}_{#2}}}

\def\sxbullet{{\raise 2pt\hbox{\bf .}}}

\begin{document}

\title[Groups with no parametric Galois extension]{\hskip 12mm Groups with no parametric \hskip 15mm Galois extension}

\author{Pierre D\`ebes}

\email{Pierre.Debes@math.univ-lille1.fr}

\address{Laboratoire Paul Painlev\'e, Math\'ematiques, Universit\'e Lille 1, 59655 Villeneuve d'Ascq Cedex, France}

\subjclass[2010]{Primary 12F12, 11R58, 14E20,  ; Secondary 14E22, 12E30,11Gxx}

\keywords{Galois extensions, inverse Galois theory, specialization, parametric extensions, twisting} 

\date{\today}

\thanks{{\it Acknowledgment}. This work was supported in part by the Labex CEMPI  (ANR-11-LABX-0007-01).}

\begin{abstract} 
We disprove a strong form of the Regular Inverse Galois Problem: there exist finite groups $G$ which do not have a realization $F/\Qq(T)$ that induces all Galois extensions $L/\Qq(U)$ of group $G$ by specializing $T$ to $f(U) \hskip -1mm \in \hskip -1mm \Qq(U)$. For these groups, we produce two extensions $L/\Qq(U)$ that cannot be simultaneously induced, thus even disproving a weaker Lifting Property.
Our examples of such groups $G$ include symmetric groups $S_n$, $n\geq 7$, infinitely many ${\rm PSL}_2(\Ff_p)$, the Monster.
Two variants of the question with $\Qq(U)$ replaced by $\Cc(U)$ and $\Qq$ are answered similarly, 
the second one under a diophantine  ``working hypothesis'' going back to a problem of Schinzel. 
We introduce two new tools: a comparizon theorem between the invariants of an extension $F/\Cc(T)$ and those obtained by specializing $T$ to $f(U) \hskip -1mm \in \hskip -1mm \Cc(U)$;
and, given two regular Galois extensions of $k(T)$, a finite set of polynomials $P(U,T,Y)$ that 
say whether these extensions have a common specialization $E/k$.
\end{abstract}

\maketitle


\section{Introduction} \label{sec:intro}

Given two fields $k\subset K$, a finite Galois extension $F/k(T)$
and a point $t_0\in \Pp^1(K)$, there is a well-defined notion of {\it specialized extension $F_{t_0}/K$}
(see {\it Basic terminology}). 
If $F$ is the splitting field over $k(T)$ of a polynomial $P\in k[T,Y]$, monic in $Y$, irreducible in $\overline{k}[T,Y]$ and $t_0$ not a root of the discriminant $\Delta_P\in k[T]$ of $P$ \hbox{w.r.t} $Y$, $F_{t_0}$ is the splitting field over $K$ of the polynomial $P(t_0,Y)$. We are mostly interested in the situations $K=k$ and $K=k(U)$ (with $U$ a new indeterminate).

The specialization process has been much studied 
towards the {\it Hilbert irreducibility} issue of existence of specializations $t_0\in k$ preserving the Galois group. Investigating the 
set, say ${\mathcal Sp}_K(F/k(T))$, of all specialized extensions $F_{t_0}/K$ with $t_0\in \Pp^1(K)$ is a further goal. 
For $k=K=\Qq$, \cite{IsrJ2} shows for example that the number of extensions 
$F_{t_0}/\Qq$ of group $G=\Gal(F/\Qq(T))$ and discriminant $|d_E|\leq y$ grows at least 
like a power of $y$, for some positive exponent, thereby proving for $G$ the ``lower bound part'' of
 a  conjecture of Malle.

Little was known on an even more fundamental question: whether ${\mathcal Sp}_K(F/k(T))$ can contain all Galois extensions $E/K$ of group contained in $G=\Gal(F/k(T))$; we then say that $F/k(T)$ is {\it $K$-parametric}, as for example $\Qq(\sqrt{T})/\Qq(T)$. 
Strikingly no group was known yet {\it not to have} a {$\Qq$-parametric or a $\Qq(U)$-parametric extension} $F/\Qq(T)$ 
while 
only four: $\{1\}$, $\Zz/2\Zz$, $\Zz/3\Zz$, $S_3$, are known to have one. No group with no $\Cc(U)$-parametric extension $F/\Cc(T)$ was even known, while only a few more with one are: cyclic groups, dihedral groups $D_{2n}$ with $n$ odd.

\vskip 2,5mm

\noindent
1.1. {\bf Groups with no $K$-parametric extension $F/k(T)$.} 
We produce many such groups: 
\vskip 0,5mm 

\noindent 
a) {\it $k=\Cc$ and $K=\Cc(U)$: 
non cyclic nilpotent groups $G$ of odd order,
symmetric groups $S_n$ with $n\geq 5$, alternating groups $A_n$ with $n \geq 6$, linear groups ${\rm PSL}_2(\Ff_p)$ with $p>7$ prime, all sporadic groups, etc.
\vskip 0,8mm 

\noindent
{\rm b)} $k=\Qq$ and $K=\Qq(U)$: the same $S_n$ and $A_n$ except for $n=6$, the ${\rm PSL}_2(\Ff_p)$ with $(\frac{2}{p})=(\frac{3}{p})=-1$, the Monster $M$, etc.
\vskip 0,8mm 

\noindent
{\rm c)} $k=K=\Qq$: 
the same last groups, under some ``working hypothesis''.}
\vskip 1,5mm

\noindent
We say more about the ``working hypothesis'' in \S 1.4 below and full statements are in  \S \ref{ssec:non-param}-\ref{ssec:non-Q-param-extensions}.
\vskip 0,5mm

These results fit in the framework of Inverse Galois Theory, a prominent open problem of which is the {\it Regular Inverse Galois Problem}: is every finite group the Galois group of some extension
$F/\Qq(T)$ that is $\Qq$-regular, \hbox{\it i.e.} $F\cap \overline \Qq = \Qq$? Possessing a $\Qq(U)$-parametric extension $F/\Qq(T)$ is for a group $G$ a strong variant. Our results show that this strong variant fails and that conditionally so does the weaker $\Qq$-parametric analog, 
thereby setting boundaries for inverse Galois theory over $\Qq$, a topic where few general statements were available. Narrowing these boundaries further,  \hbox{\it e.g.} removing ``conditionally'' in the version over $\Qq$, still remains desirable. We note this weaker but unconditional result\footnote{Remark \ref{rem:on-corollary-with-WH} explains how Legrand's result can be deduced from ours under our working hypothesis.} of Legrand \cite{legrand-at-least-1-nonparam}: every non trivial group that has at least one $\Qq$-regular realization
$F/\Qq(T)$ has one that is not $\Qq$-parametric.

\vskip 2,5mm

\noindent
1.3. {\bf The Lifting Property.} Our best result is in fact stronger than the non-existence of  parametric extensions and may also be more informative, in that it shows better the obstruction to having a parametric extension that our method reveals, which is not the absence of regular realizations $F/k(T)$ but the existence of several that cannot be ``simultaneously lifted''. Specifically, for every group $G$ in list a) above with $k= \Cc$, or, in list b) with $k\subset \Cc$, excluding $G=A_n$\footnote{For $G=A_n$, the two extensions $L_i/k(U)$ from (*) should be replaced by three.}, we show that

%

\vskip 1,5mm

\noindent
{\rm (*)} {\it there exist two $k$-regular Galois extensions $L_1/k(U)$ and $L_2/k(U)$ of group $G$ with this property: there is no $k$-regular Galois extension $F/k(T)$ of group $G$ such that $F\Cc/\Cc(T)$ specializes to $L_1\Cc/\Cc(U)$ and $L_2\Cc/\Cc(U)$ at two points $T_{01}, T_{02}\in \Cc(U)$. 
}
\vskip 1,5mm


%
%
%
In geometrical terms, this shows that the following {\it Lifting Property} 
\vskip 1,5mm

\noindent
 (${\rm LP}_k(G)$) {\it any $N$ $k$-G-Galois covers of $\Pp^1_k$ of group $G$ can be, after scalar extension to $\Cc$, obtained by pull-back along a map $\Pp^1_\Cc\rightarrow \Pp^1_\Cc$ from some $k$-G-Galois cover $f:X\rightarrow \Pp^1_k$ of group $G$}. 
 \vskip 1,5mm
 
 \noindent
fails for $N=2$. Thus in the chain of implications (for each $N\geq 2$):
\vskip 1,5mm

\hskip 8mm 
$\begin{matrix}
\hbox{\it $G$ has a} \\
\hbox{\it $k(U)$-parametric} \\
\hbox{\it extension $F/k(T)$} \\
\end{matrix}
\hskip 2mm \Rightarrow \hskip 2mm 
\begin{matrix}
$\hbox{\it ${\rm LP}_k(G)$}$ \\
\hbox{{\it holds for} $N$} \\
\end{matrix}
\hskip 2mm \Rightarrow\hskip 2mm 
\begin{matrix}
\hbox{\it $G$ is a regular} \\
\hbox{\it Galois group} \\
\hbox{\it over $k$} \\
\end{matrix}
$
\vskip 1mm

\noindent
not only the first condition fails, but also the second one.

Our Lifting Property further relates to some variant investigated by Colliot-Th\'el\`ene \cite{colliot-annals}, about which he also obtains a negative conclusion, for some $p$-group $G$ over some ``large'' field and he observes that ``other examples remain to be seen''.

\vskip 2,5mm

\noindent
1.4. \hbox{\bf parametric \hbox{\it vs.} generic.} As a consequence of our results, the groups from list \S 1.1 a) do not have a {\it generic} extension $F/\Cc(T)$; indeed {\it ge\-ne\-ric} is a stronger notion meaning 
``$L$-parametric for all fields $L\supset \Cc$''.
This was already known 
by a result of Buhler-Reichstein \cite{buhler-reichstein}: the only  groups to have a generic extension $F/\Cc(T)$ are the cyclic groups and dihedral groups $D_{2n}$ with $n$ odd. 
Our non $\Cc(U)$-parametric
con\-clu\-sion however is stronger: the extensions to be para\-me\-tri\-zed in the generic context include all Galois extensions $E/L$ of group $G$ with $L$ {\it any field containing $\Cc$} and
it readily follows that $G$ should then be a subgroup of ${\rm PGL}_2(\Cc)$ \cite[prop.8.14]{JLY}. This is an important preliminary reduction for generic extensions that can no longer be used
 if $F/\Cc(T)$ is only $\Cc(U)$-parametric (\hbox{\it i.e.}  only parametrizes extensions $E/\Cc(U)$). There exist in fact groups that have a $\Cc(U)$-parametric extension but no generic extension $F/\Cc(T)$ (corollary \ref{cor:subgroups-PGL2C}).

\vskip 2,5mm

\noindent
1.5. {\bf The comparizon theorem.}
We will first focus on the situation $K=k(U)$ with $k\subset \Cc$. Results mentioned above follow from a general criterion (criterion \ref{crit:non-param-gen-criterion}) for some set of $k$-regular Galois extensions $L/k(U)$ of group $G$ to be $k(U)$-specializations of a $k$-regular Galois extension $F/k(T)$ of group $G$. A main point of our approach is that
\vskip 1mm

\noindent
(*) {\it the branch point number of an extension\footnote{The extension $F/\Cc(T)$ need not be assumed to be Galois in this statement.} $F/\Cc(T)$ cannot drop under specialization of $T$ in $\Cc(U)$, unless $F/\Cc(T)$ is one from a list of exceptional extensions with $F$ of genus $0$ {\rm (see  theorem \ref{thm:order-invariants} (a))}.}
\vskip 1mm

\noindent
Despite its basic nature, this did not seem to be known; the difficulty is that the group may drop and that the ramification of the specialization point $T_0\in \Cc(U)$ may cancel some of the ramification of $F/\Cc(T)$. 
We prove a more precise version giving better estimates 
of the branch point number and other invariants of specialized extensions 
$F_{T_0}/\Cc(U)$, $T_0\in \Cc(U)$, which could be interesting beyond this paper  (theorem \ref{thm:r-in-spec}).
 
\vskip 2,5mm

\noindent
1.6. {\bf A pre-order on Galois extensions $L/\Cc(T)$.}
The situation $K=\Cc(U)$ has another interesting feature: specialized extensions $F_{T_0}/\Cc(U)$ with $T_0\in \Cc(U)$ remain extensions of the rational function field in one indeterminate, as the initial extension $F/\Cc(T)$. The specialization process induces a (partial) pre-order on the set of Galois extensions $L/\Cc(T)$. We will show that this is in fact an order on a big subset (see theorem \ref{thm:order-invariants} (b)), with this consequence: 
\vskip 1mm

\noindent
(*) {\it for ``most'' groups $G$ {\rm(}e.g. all groups of rank $\geq 4${\rm )}, there is at most one $\Cc(U)$-parametric extension $F/\Cc(T)$ of group $G$.}
\vskip 1mm

The pre-order that we use to investigate the minimal elements raises further questions about the ordered structure of Galois extensions of $k(T)$ 
that are certainly worthwhile being studied.



%

\vskip 2,5mm

\noindent
1.7. {\bf The twisted polynomial.}
Our results in the situation that $k=K$ is a number field will  be obtained from those with $K=k(U)$ by specialization, but of the indeterminate $U$ this time. To this end we will generalize a tool introduced in \cite{IsrJ2} as the ``self-twisted cover''.
Theorem \ref{thm:twist-main} is the concrete statement that makes this specialization approach work.
%
It is interesting for its own sake: given two $k$-regular Galois extensions  $F/k(T)$, $L/k(T)$ of group $G$, it provides a finite set of polynomials $\widetilde P_F^L(U,T,Y)\in k[U,T,Y]$ which have the answer to the question of whether  $F/k(T)$ and $L/k(T)$ have a common specialization: 
 
\vskip 1mm

\noindent
(*) {\it for all  but finitely many $u_0\in k$, $L_{u_0}/k = F_{t_0}/k$ for some $t_0\in k$ not a branch point of $F/k(T)$ if and only if one of the polynomials  $\widetilde P_F^L(u_0,t_0,Y)$ has a root $y_0\in k$; and similarly, $L_{U}/k(U) = F_{T_0}/k(U)$ for some $T_0\in k(U)$ iff one polynomial  $\widetilde P_F^L(U,T_0,Y)$ has a root $Y_0\in k(U)$.}
\vskip 1,5mm

\noindent
The working hypothesis, which goes back to some diophantine problem of Schinzel, relates the absence of $k(U)$-rational points $(T_0,Y_0)\in k(U)^2$ on each of the curves $\widetilde P_F^L(U,T,Y)=0$ to the absence, for infinitely many $u_0\in k$, of $k$-rational points $(t_0,y_0)\in k^2$ on each of the curves $\widetilde P_F^L(u_0,T,Y)=0$ (see \S \ref{ssec:working-hypothesis}), thereby extending Hilbert's Irreducibility Theorem to polynomials with two indeterminates and one parameter. It has no known counter-example.

\vskip 4mm


The paper is organized as follows. \S \ref{sec:presentation} presents in full detail the results of our paper. We reduce their proofs to that of two main theorems: the comparizon theorem \ref{thm:order-invariants} and the ``twisting'' theorem \ref{thm:twist-main}. We state them and explain their implications. Their proofs, which are rather independent, are given in \S \ref{sec:C(U)-specializations} and \S \ref{sec:twisting}. Finally \S \ref{sec:preliminaries} is an appendix where we have collected a few classical results that enter 
in our proofs and that we have rephrased to fit our field arithmetic set-up;  this section is used in 
 \S \ref{sec:C(U)-specializations} and in  \S \ref{sec:twisting}.
  We start below with some basic terminology.


\vskip 3mm

\noindent
{\fifteenpc Basic terminology} (for more details, see \cite{DeDo1} or \cite{DeLe1}).

The base field $k$ is always assumed to be of characteristic $0$. Is also fixed a big algebraically closed field containing the complex field $\Cc$ and the indeterminates that will be used and in which all field compositum should be understood.

Given a field $K$, an extension $F/K(T)$ is said to be \textit{\textbf {$K$-regular}} 
if $F\cap \overline K= K$. We make no distinction between a $K$-regular extension $F/K(T)$ and the associated $K$-regular cover $f:X\rightarrow \Pp^1$: $f$ is the normalization of $\Pp^1_K$ in $F$ and $F$ is the function field $K(X)$ of $X$. 
The ``field extension'' viewpoint is mostly used in this paper.


We also use \textit{\textbf {affine equations}}: 
we mean the irreducible polynomial $P\in K[T,Y]$ of a primitive element of $F/K(T)$, \hbox{integral over $K[T]$.}

By \textit{\textbf {group}} and \textit{\textbf {branch point set}}
of a $K$-regular extension $F/K(T)$, we mean 
those of the extension  $F\overline K/\overline K(T)$: the group of $F\overline K/\overline K(T)$ is the Galois group of its Galois closure. The branch point set of $F\overline K/\overline K(T)$ is the (finite) set of points $t\in \Pp^1(\overline K)$ such that the associated discrete valuations are ramified in  $F/\overline K(T)$.

The field $K$ being of characteristic $0$, we also use the 
\textit{\textbf {inertia canonical invariant}}\footnote{This is also called ``branching type'' by some authors.} ${\bf C}$ of the $K$-regular extension $F/K(T)$, defined as follows. If ${\bf t}=\{t_1,\ldots,t_r\}$ is the branch point set of $f$, then ${\bf C}$ is a $r$-tuple 
$(C_1,\ldots,C_r)$ of conjugacy classes of the group $G$ of $f$: for $i=1,\ldots,r$, $C_i$ is the conjugacty class of the distinguished\footnote{``distinguished'' means that these generators correspond to the $e_i$th root $e^{2i\pi/e_i}$ of $1$ in the canonical isomorphism  $I_{\frak P} \rightarrow \mu_{e_i} =\langle e^{2i\pi/e_i} \rangle$.} generators of the inertia groups $I_{\frak P}$ above $t_i$ in the Galois closure $\widehat F/K(T)$ of $F/K(T)$. 

We also use the notation ${\bf e}=(e_1,\ldots,r_r)$ for the $r$-tuple
with $i$th entry the ramification index $e_i=|I_{\frak P}|$ of primes above $t_i$; $e_i$ is also the  order of elements of $C_i$, $i=1,\ldots,r$.

We say that two $K$-regular extensions $F/K(T)$ and $L/K(T)$ are \textit{\textbf {isomorphic}} 
if there is a field isomorphim $F\rightarrow L$ that restricts to an automorphism $\chi: K(T)\rightarrow K(T)$ equal to the identity on $K$ and that they are \textit{\textbf {K{\rm (}\hskip -2pt T\hskip -0pt{\rm )}-isomorphic}} if in addition $\chi$ is the identity on $K(T)$.

Given a 
Galois extension $F/K(T)$ and $t_0\in {\mathbb P}^1(K)$, 
the \textit{\textbf {specialization of $F/K(T)$ at $t_0$}} 
is the Galois extension $F_{t_0}/K$ defined as follows. Consider 
the localized ring $A_{t_0}=K[T]_{\langle T-t_0\rangle}$ of $K[T]$ at $t_0$, the integral closure $B_{t_0}$ of 
$A_{t_0}$ in $F$. Then $F_{t_0}/K$ the residue extension of an arbitrary prime 
ideal of $B_{t_0}$ above $\langle T-t_0\rangle$. (As usual use the local ring 
$K[1/T]_{\langle 1/T\rangle}$ and its ideal $\langle 1/T\rangle$ if $t_0=\infty$). 

If $P\in K[T,Y]$ is an affine equation of $F/K(T)$ and $\Delta_P\in K[T]$ is its 
discriminant w.r.t. $Y$, then for every $t_0\in K$ such that $\Delta_P(t_0)\not=0$, 
$t_0$ is not a branch point of $F/K(T)$ and the specialized extension $F_{t_0}/K$
is the splitting field over $K$ of $P(t_0,Y)$.

If $K^\prime$ is a field containing $K$, the \textit{\textbf{specialization $F_{t_0}/K^\prime$ of $F/K(T)$ at $t_0$}}
is the extension $(FK^\prime)_{t_0}/K^\prime$. If $K^\prime=K(U)$, $T_0\in K(U)$ is a non-constant
rational function\footnote{We use a capital letter for the specialization point $T_0$ to stress that it is a function 
$T_0(U)$ contrary to the situation for which it is a point in the ground field and the notation 
$t_0$ is preferred.} and
$P\in K[T,Y]$ is an affine equation of $F/K(T)$, then $\Delta_P(T_0)\not=0$ and so 
$P(T_0(U),Y)$ is an affine equation of the specialized extension $F_{T_0}/K(U)$.

If the extension $F/K(T)$ is not Galois, the above definition leads to several specializations $F_{t_0}/K$: the prime
ideals of $B_{t_0}$ above $\langle T-t_0\rangle$ are not conjugate in general. When we use this extended definition (only once in theorem \ref{thm:r-in-spec} (a)), we will talk about {\it a} specialization instead of {\it the} specialization $F_{t_0}/K$.

We finally recall the \textit{\textbf{Riemann Existence Theorem}} (RET) which indicates that Galois extensions $F/k(T)$ are well-understood if $k$ is algebraically closed and which we will use in this practical form.
\vskip 1,5mm

\noindent 
{\bf Riemann Existence Theorem.} {\it Given a group $G$, an integer $r\geq 2$, a subset ${\bf t}\subset \Pp^1(\Cc)$ of $r$ points and an $r$-tuple ${\bf C} = (C_1,\ldots,C_r)$ of conjugacy classes of $G$, there is a Galois extension $F/\Cc(T)$  of group $G$, branch point set ${\bf t}$ and inertia canonical invariant ${\bf C}$ iff there exists $(g_1,\ldots,g_r)\in C_1\times \cdots \times C_r$ such that  
$g_1 \cdots g_r = 1$ and $\langle g_1,\ldots, g_r \rangle=G$. Furthermore the number of such extensions $F/\Cc(T)$ {\rm (}in a fixed algebraic closure $\overline{k(T)}${\rm )} equals the number of $r$-tuples $(g_1,\ldots,g_r)$ as above, counted modulo componentwise conjugation by an element of $G$.}
\vskip 1mm


\section{Main results} \label{sec:presentation}
We present our main results: the specialization process and the associated
 order in the situation $k=\Cc$ and $K=\Cc(U)$ (\S \ref{ssec:order}), some new examples of groups 
 with a $\Cc(U)$-parametric extension $F/\Cc(T)$ (\S \ref{ssec:param}), a method to produce groups 
 with no $k(U)$-parametric extension $F/k(T)$  (\S \ref{ssec:non-param}), our
 ``twisted polynomial'' $\widetilde P_F^L(U,T,Y)$ and its use towards the construction of 
 groups with no $k$-parametric extension $F/k(T)$ (\S \ref{ssec:non-Q-param-extensions}).

\subsection{$\Cc(U)$-specializations of Galois extensions $F/\Cc(T)$} \label{ssec:order}
This subsection gives the main definitions and our first main tool (theorem \ref{thm:order-invariants}).

\subsubsection{Comparizon theorem} \label{ssec:comparizon}
Given a $K$-regular extension $F/K(T)$, we use the following notation for its {\it invariants}: $G_F$ for the group, $r_F$ for the branch point number, ${\bf C}_F$ for the inertia canonical invariant 
and $g_F$ for the genus of $F$; they are invariant inside the isomorphism class of $F/K(T)$. 

Given two Galois extensions $F/\Cc(T)$ and $L/\Cc(T)$, we write  
\vskip 0,8mm

\centerline{$F/\Cc(T) \specz L/\Cc(T)$}
\vskip 1,2mm

\noindent 
if $L/\Cc(U)$ is the specialization $F_{T_0}/\Cc(U)$ of $F/\Cc(T)$ at some non-constant rational function $T_0\in \Cc(U)$.

For a conjugacy class $C$ of a group $G$, set $C^\Zz =\bigcup_{\alpha\in \Zz} C^\alpha$;  $C^\Zz$ corresponds to the conjugacy class of the cyclic subgroup generated by any element of $C$. 
%
Given tuples ${\bf C}=(C_1,\ldots,C_r)$ and ${\bf C}^\prime=(C_1^\prime,\ldots,C_r^\prime)$ of conjugacy classes of $G$ and $G^\prime$, write ${\bf C} \prec {\bf C}^\prime$
%
%
 if for every $j\in \{1,\ldots,r^\prime\}$, there exists $i\in \{1,\ldots,r\}$ such that $C^\prime_j \subset C_i^\Zz $.
 

\begin{theorem} \label{thm:order-invariants}
{\rm (a)} Let $F/\Cc(T)$ and $L/\Cc(T)$ be two finite Galois extensions. Assume $g_F\geq 1$. Then we have:
\vskip 1,5mm

\centerline{$F/\Cc(T) \specz L/\Cc(T) \Rightarrow (G_F,r_F,{\bf C}_F) \prec (G_L,r_L,{\bf C}_L)$}
\vskip 1,5mm

\noindent
where the right-hand side condition means that $G_F\supset G_L$, $r_F \leq r_L$ and ${\bf C}_F \prec {\bf C}_L$. If in addition $G_F=G_L$, the impli\-cation also holds if $g_F=0$; and we have
$g_F\leq g_L$ if $r_F\geq 4$.\end{theorem}


As recalled in \S \ref{ssec:exceptional-list}, the excluded case $g_F=0$ is known to only happen when $r_F\leq 3$ and $G_F$ is a subgroup of ${\rm PGL}_2(\Cc)$, \hbox{\it i.e.}, one of these groups: 
$\Zz/n\Zz$ {\rm (}$n\geq 1${\rm )}, $(\Zz/2\Zz)^2$, $A_4$, $S_4$, $A_5$, $D_{2n}$ {\rm (}$n\geq 3${\rm )}.
For each such group $G_F$, there is, up to isomorphism, only one Galois extension $F/\Cc(T)$ of group $G_F$ and genus $g_F =0$.

Theorem \ref{thm:order-invariants} will be deduced from theorem \ref{thm:r-in-spec} which offers more precise estimates, for example, the lower bound 
\vskip 1,5mm

\noindent
(*) \hskip 30mm $r_{L} \geq (N-4)\hskip 1pt r_F + 4$ 

\vskip 1,5mm

\noindent
if $L/\Cc(T) = F_{T_0}/\Cc(T)$ with $T_0\in \Cc(T)$ of degree $N$.

\subsubsection{The order $\prec$} \label{sssec:order}
These estimates will further show that, as stated below, the pre-order $\prec$ is antisymmetric on a big subset of all Galois extensions, regarded modulo isomorphisms.
 %
%

Specifically, denote by ${\mathcal G}^\ast$ the set of groups that are 
\vskip 1mm

\noindent
(*) {\it {\rm (}of rank $\geq 4${\rm )} {\rm or} {\rm (}or rank $3$ and odd order{\rm )} {\rm or}  {\rm (}of rank $2$ and order not divisible by $2$ or $3${\rm )} {\rm or}  {\rm (}a subgroup of ${\rm PGL}_2(\Cc)${\rm )} }
\vskip 1mm

\noindent
and by ${\mathcal E}^\ast$ the set of all Galois extensions $F/\Cc(T)$, viewed up to isomor-  phism such that ({$G_F\in {\mathcal G}^\ast, G_F\not\subset {\rm PGL}_2(\Cc)$) {\it or} ($g_F=0$).

The notion of ``parametric extensions'' appearing below was introduced in \S \ref{sec:intro}; the definition is recalled right next in \S \ref{ssec:basics-parametric-extensions}.


\vskip 2mm

\noindent
{\bf Theorem \ref{thm:order-invariants}.} 
{\rm (b)} {\it The relation $\specz$ induces a {\rm (}partial{\rm )} order on ${\mathcal E}^\ast$. Consequently, for every group $G\in {\mathcal G}^\ast$, there is at most one Galois extension $F/\Cc(T)$ of group $G$ that is 
$\Cc(U)$-parametric.}
\vskip 2mm
The uniqueness part follows from
the first part: the main point is that if an extension $F/\Cc(T)$ is $\Cc(U)$-parametric of group $G\in {\mathcal G}^\ast$, it is {\it the} smallest  (for $\specz$) Galois extension $L/\Cc(T)$ of group $G$.\footnote{For a Galois extension $L/\Cc(T)$ of group $G$, there is a Galois extension $F/\Cc(T)$ such that $F/\Cc(T) \prec L/\Cc(T)$ and $F/\Cc(T)$ is minimal (for $\prec$) among all Galois ex\-ten\-sions of $\Cc(T)$ of group $G$. Several such extensions $F/\Cc(T)$ exist in general.
}

We have no example of two non-isomorphic Galois extensions $F/\Cc(T)$ and $L/\Cc(T)$ such that $F/\Cc(T)\prec L/\Cc(T)$ and $L/\Cc(T)\prec F/\Cc(T)$, and in particular, no example of a group $G$ 
that has two $\Cc(U)$-para\-me\-tric extensions $F/\Cc(T)$. In fact the groups that are known to have at least one 
$\Cc(U)$-parametric extension $F/\Cc(T)$ are the finite subgroups of ${\rm PGL}_2(\Cc)$, and for them, uniqueness is part of theorem \ref{thm:order-invariants} (for the existence, see corollary \ref{cor:subgroups-PGL2C}).

The proofs of the two parts of theorem \ref{thm:order-invariants} are given in \S \ref{sec:branch-point-number-growth}
and \S \ref{ssec:proof-continued}.

 


\subsubsection{Parametric extensions} \label{ssec:basics-parametric-extensions} The following definition 
was introduced 
by F.~Legrand \cite{legrand-thesis}, \cite {legrand1}, \cite{legrand2}. Close variants exist in connection with
 the notion of generic polynomials \cite{JLY}. 

\begin{definition} \label{def:parametric-extension} A finite $k$-regular Galois extension $F/k(T)$ of group $G$ 
is {\it $k$-parametric} if for every Galois extension $E/k$ of group contained in $G$, there exists $t_0 \in \Pp^1(k)$, not a branch point of $F/k(T)$, such that the specialized extension $F_{t_0}/k$ is $k$-isomorphic to $E/k$. Given an overfield $K\supset k$, $F/k(T)$ is {\it $K$-parametric} if $FK/K(T)$ is $K$-parametric. The group $G$ is then said {\it to have a $K$-parametric extension $F/k(T)$}.
\end{definition}

The extension $\Qq(\sqrt{T})/\Qq(T)$ is the standard exemple of an extension $F/\Qq(T)$ that is $K$-parametric; it is for all fields $K \supset \Qq$ and so is in fact generic. Recall indeed that ``generic'' for a finite $k$-regular Galois extension $F/k(T)$ means ``$K$-parametric for all fields $K \supset k$''.


%

\begin{remark}

\vskip 1mm
%
%
(a) {\it A $k(U)$-parametric extension $F/k(T)$ is  $k$-parametric.}



\begin{proof} Let $F/k(T)$ be a $k(U)$-parametric extension of group $G$. The extension $Fk(U)/k(U,T)$ is $k(U)$-regular and {\it a fortiori} the extension $F/k(T)$ is $k$-regular. Let $E/k$ be a Galois extension of group $H\subset G$.
As $F/k(T)$ is $k(U)$-parametric, there is $T_0\in k(U)$ such that the specialized extension $F_{T_0}/k(U)$ is $k(U)$-isomorphic to $E(U)/k(U)$. Hence for all but finitely many $u_0\in \Pp^1(k)$, the extension, $(F_{T_0})_{u_0}/k$, obtained by specializing  $F_{T_0}/k(U)$ at $u_0$ is $E/k$. The conclusion follows since, as explained below, for all but finitely many $u_0\in \Pp^1(k)$, $(F_{T_0})_{u_0}/k$ is also the specialized extension $F_{T_0(u_0)}/k$. 

This is clear if $T_0\in k$.  Assume $T_0\notin k$ and let $P\in k[T,Y]$ be an affine equation of $F/k(T)$.  Then $F_{T_0}$ is 
the splitting field over $k(U)$ of $P(T_0(U),Y)$ and, as $F/k(T)$ is Galois, it is also the splitting field of any irreducible factor $Q\in k[U,Y]$ of $P(T_0(U),Y)$. Thus such a $Q$ is an affine equation of the Galois extension $F_{T_0}/k(U)$. For all but finitely many $u_0\in \Pp^1(k)$, the extension $(F_{T_0})_{u_0}/k$
is the splitting field over $k$ of $Q(u_0,Y)$ and also of $P(T_0(u_0),Y)$.
This concludes the argument as for all but finitely many $u_0\in \Pp^1(k)$, $F_{T_0(u_0)}/k$ is also 
the splitting field over $k$ of $P(T_0(u_0),Y)$. \end{proof}
\vskip -2mm

This argument applies inductively to show that condition ``$F/k(T)$  is $k(U_1,\ldots,U_s)$-parametric'' is stronger and stronger as $s$ gets bigger; it remains however always weaker than ``generic''. 
\vskip 2mm

\noindent
(b) On the other hand, for a $k$-regular Galois extension $F/k(T)$ and an algebraic extension $E/K$ with $K\supset k$, the connection between ``$E$-parametric'' and ``$K$-parametric'' is not so clear. As we will see, our criterion to produce non $k(U)$-parametric extensions is all the more efficient that there are more $k(U)$-regular realizations of the group $G$ in question, and so will be more fruitful when $k$ is algebraically closed. 
We however do not have any proof of any implication.
\end{remark}

%
%

\subsection{Groups with a $k(U)$-parametric extension $F/k(T)$} \label{ssec:param}
We have the following statement.

\begin{corollary} \label{cor:subgroups-PGL2C}
All subgroups of ${\rm PGL}_2(\Cc)$:
\vskip 1mm
 
\centerline{$\Zz/n\Zz$ {\rm (}$n\geq 1${\rm )}, $(\Zz/2\Zz)^2$, $A_4$, $S_4$, $A_5$, $D_{2n}$ {\rm (}$n\geq 3${\rm )}}
\vskip 1mm

\noindent
have a $\Cc(U)$-parametric extension. Out of them, $\Zz/n\Zz$ with $n=1,2,3$ and $D_6 = S_3$ have a $k(U)$-parametric extension for every field $k$ of characteristic $0$.
\end{corollary}

Theorem \ref{thm:order-invariants} (b) shows further that the $\Cc(U)$-parametric extension claimed to exist is unique up to isomorphism. 

\begin{proof}
The first part is a consequence of corollary \ref{cor:genus0-param}; the main points are the ``twisting lemma'' 
and Tsen's theorem. The four groups in the second part are known to have a generic extension $F/\Qq(T)$ \cite{JLY}
\end{proof}

\begin{remark}[Parametricity and genericity] \label{rem:param-gen}
Cyclic groups and dihedral groups $D_{2n}$ with $n$ odd were known to have a $\Cc(U)$-parametric extension as they have a generic extension $F/\Cc(T)$: for $\Zz/d\Zz$, take $F=\Cc(T^{1/d})/\Cc(T)$ ($d\geq 1)$; for $D_{2n}$, 
 it is a result of Hashimoto-Miyake \cite{hashimoto-miyake} (see also \cite[theorem 5.5.4]{JLY}). These groups are the only ones to have a generic extension $F/\Cc(T)$ \cite{buhler-reichstein}.
The other subgroups of ${\rm PGL}_2(\Cc)$: $(\Zz/2\Zz)^2$, $A_4$, $S_4$, $A_5$,  $D_{2n}$ with $n$ even, have a $\Cc(U)$-parametric extension but no generic extension $F/\Cc(T)$. Whether subgroups of ${\rm PGL}_2(\Cc)$ other than $\Zz/n\Zz$ with $n=1,2,3$ and $S_3$ have a $\Qq(U)$-parametric extension $F/\Qq(T)$ is unclear.\footnote{Even if for some of these groups ($(\Zz/2\Zz)^2$, $S_4$, $D_{2n}$ with $n$ even), the unique $\Cc(U)$-parametric extension $F/\Cc(T)$ is defined over $\Qq$ (\S \ref{ssec:exceptional-list}), a $\Qq$-model $F_0/\Qq(T)$ is 
not guaranteed to be $\Qq(U)$-parametric: although any extension $L/\Qq(U)$ of group $G$ is a specialization of $F/\Cc(T)$, the specialization point $T_0$, which is in $\Cc(U)$ may not be in $\Qq(U)$. Anticipating on \S \ref{ssec:twisting}, the issue relates to the following: a polynomial equation $P(U,T,Y)=0$ with $P\in \Qq[U,T,Y]$
may have a solution $(T_0(U),Y_0(U))\in \Cc(U)^2$ but no solution in $\Qq(U)^2$: think of $Y^2+T^2+U^2+1=0$.}
\end{remark}

\subsection{Groups with no $k(U)$-parametric extension $F/k(T)$} \label{ssec:non-param}
We explain how we use theorem \ref{thm:order-invariants} to produce groups with no $k(U)$-parametric extension
 $F/k(T)$, with $k$ algebraically closed in \S \ref{ssec:examples-non-paramC} and $k$ non algebraically closed in \S \ref{ssec:examples-non-param}. We start with a general criterion in \S \ref{ssec:general-criterion}. For simplicity, assume $k\subset \Cc$; there is no loss of generality.

Our method will in fact lead to a slightly better conclusion than ``no $k(U)$-parametric extension''. To this end we define a $k$-regular Galois extension $F/k(T)$ to be {\it weakly $k(U)$-parametric} of group $G$ if for every $k$-regular Galois extension $L/k(U)$ of group $G$ (and not of group {\it contained in} $G$ as for {$k(U)$-parametric}), $L\Cc/\Cc(U)$ is a specialization $F_{T_0}/\Cc(U)$ for some $T_0\in \Cc(U)$ (while for {$k(U)$-parametric}, the requested $T_0$ is in $k(U)$). Obviously we have:
\vskip 1mm

\centerline{$k(U)$-parametric $\Rightarrow$  weakly $k(U)$-parametric}


\subsubsection{General criterion} \label{ssec:general-criterion}
Given a subfield $k\subset \Cc$ and a finite group $G$, denote the set of all $k$-regular extensions
$L/k(T)$ of group $G$ by ${\mathcal R}_k(G)$.
From theorem \ref{thm:order-invariants}, if $F/k(T)$ is a weakly $k(U)$-parametric extension of group $G$, we must have
$(G,r_F,{\bf C}_F) \prec (G,r_L,,{\bf C}_L)$ for every extension $L/k(T)\in {\mathcal R}_k(G)$.
The general idea is to show that there is no Galois extension $F/\Cc(T)$ of group $G$ such that
\vskip 1mm

\noindent
(*) \hskip 10mm {\it $r_F \leq r_L$ and ${\bf C}_F \prec {\bf C}_F$ for every $L/k(T)\in {\mathcal R}_k(G)$.}
 
 \vskip 1mm
 
Criterion \ref{crit:non-param-gen-criterion} below uses the following additional notation. Say that two conjugacy classes $C$ and $C^\prime$ of $G$ are {\it very different}, and write then $C \verydiff C^\prime$ if there is no conjugacy class $C_0$ such $C\subset C_0^\Zz$ and $C^\prime \subset C_0^\Zz$. For example, if $C$ is the conjugacy class of a generator of a maximal cyclic subgroup of $G$, then $C \verydiff C^\prime$ if and only if $C^\prime \not\subset C^\Zz$. In particular,
if $C^\prime$ is also the conjugacy class of a generator of a maximal cyclic subgroup of $G$, then $C \verydiff C^\prime$ if and only if $C^\Zz\not=(C^\prime)^\Zz$, {\it i.e.}, if the two maximal cyclic subgroups associated with $C$ and $C^\prime$ are not conjugate in $G$. Many concrete examples appear in \S \ref{ssec:examples-non-paramC} and \S \ref{ssec:examples-non-param} below.



%
%

%
%

\begin{criterion} \label{crit:non-param-gen-criterion}  Let ${\mathcal R}$ be a nonempty subset of ${\mathcal R}_k(G)$. Let $\rho_{\mathcal R}$ be the minimum number $r_L$ for some $L/k(T)\in {\mathcal R}$. Assume
the list of con\-ju\-g\-acy clas\-ses $C$ appearing in some tuple ${\bf C}_L$ with $L/k(T)\in {\mathcal R}$ contains at least $\nu_{\mathcal R}$ of them that are pairwise very different, and that $\nu_{\mathcal R} > \rho_{\mathcal R}$. 
\vskip 1 mm

\noindent
{\rm (**)} {\it  Then there is no $k$-regular Galois extension $F/k(T)$ of group $G$ that admits each extension $L\Cc_/\Cc(U)\in {\mathcal R}$ as a specialization $F_{T_0}/\Cc(U)$ for some $T_0\in \Cc(U)$ {\rm (}depending on $L\Cc/\Cc(U)${\rm )}.}
\vskip 1 mm

\noindent
In particular, $G$ has no weakly $k(U)$-parametric extension and a fortiori no 
$k(U)$-parametric extension $F/k(T)$. 
\end{criterion} 

The smaller the subset ${\mathcal R}$ is the stronger is conclusion (**), which, in the extreme case ${\mathcal R}={\mathcal R}_k(G)$, is equivalent to $G$ not having a  weakly $k(U)$-parametric extension $F/k(T)$. 

\begin{proof} Assume that 
there is a $k$-regular Galois extension $F/k(T)$ of group $G$ such that  $F\Cc/\Cc(T)$ specializes to each of the extensions $L\Cc/\Cc(T)$ with $L/k(T)\in {\mathcal R}$. It follows from theorem \ref{thm:order-invariants} 
 that $r_F \leq \rho_{\mathcal R}$ and ${\bf C}_F \prec {\bf C}_L$ for every $L/k(T)\in {\mathcal R}$.
%
%
Hence if $C$, $C^\prime$ are two conjugacy classes appearing in the list of tuples ${\bf C}_L$ with $L/k(T)\in {\mathcal R}$, there are conjugacy classes $C_{F,i}$, $C_{F,j}$ from ${\bf C}_F$ such that $C\subset C_{F,i}^\Zz$ and $C^\prime \subset C_{F,j}^\Zz$. If $C \verydiff C^\prime$, then $C_{F,i}\not=C_{F,j}$.
Therefore $r_F \geq \nu_{\mathcal R}$. Hence $\rho_{\mathcal R} \geq \nu_{\mathcal R}$, a contradiction.
\end{proof}

\subsubsection{Groups with no $\Cc(U)$-parametric extension} \label{ssec:examples-non-paramC} 
Denote the number of conjugacy classes of maximal cyclic subgroups of a group $G$ by $\nu(G)$
and the rank of $G$ (minimal cardinality of a generating set) by ${\rm rk}(G)$. 

%
%

\begin{corollary} \label{cor:criterion-overC} 
Assume $k$ is algebraically closed. If $\nu(G) \geq {\rm rk}(G)+2$, conclusion {\rm (**)} from criterion \ref{crit:non-param-gen-criterion}  holds  with ${\mathcal R}$ consisting of two extensions $L_1/k(T)$ and $L_2/k(T)$. Consequently $G$ has no weakly $k(U)$-pa\-ra\-me\-tric extension and a fortiori no 
$k(U)$-parametric extension $F/k(T)$.
\end{corollary}

\begin{proof} This directly follows from criterion \ref{crit:non-param-gen-criterion} applied with ${\mathcal R}$ consisting of two extensions $L_1/k(T)$ and $L_2/k(T)$ chosen so that  $r_{L_1}= {\rm rk}(G)+1$ and ${\bf C}_{L_2}$ contains all non trivial conjugacy classes of $G$. Such extensions exist thanks to the RET.
\end{proof}

As we check below, the groups in the following non exhaustive list satisfy
the condition $\nu(G)\geq {\rm rk}(G)+2$. 

\begin{corollary} \label{cor:non-param-overC}
Assume $k$ is algebraically closed. \hbox{None of these groups:}
\vskip 0,5mm

\noindent
- $S_n$, $n\geq 5$ and $A_n$, $n \geq 6$, 
\vskip 0,5mm


\noindent 
- non cyclic nilpotent groups $G$ with abelianization $G^{{\rm ab}}$ different from $\Zz/2\Zz \times \Zz/2\Zz$, 
in particular non cyclic nilpotent groups $G$ of odd order,
\vskip 0,5mm

\noindent 
-  linear groups ${\rm PSL}_2(\Ff_p)$, $p>7$ prime, 
\vskip 0,5mm

\noindent 
-  all sporadic simple groups, 
\vskip 0,5mm

\noindent
have a weakly $k(U)$-parametric extension $F/k(T)$. More precisely conclusion {\rm (**)} from criterion \ref{crit:non-param-gen-criterion} holds with $k=\Cc$ and ${\mathcal R}$ consisting of two extensions except for the groups $A_n$ for which three are needed.
\end{corollary}

On the other hand, all finite subgroups of ${\rm PGL}_2(\Cc)$ can be double-checked not to satisfy $\nu(G) \geq {\rm rk}(G)+2$ (which must also hold because they have a $\Cc(U)$-parametric extension $F/\Cc(T)$). The quaternion group ${\mathbb H}_8$ is another example. The complete list of groups satisfying the condition remains to be established. It seems that it contains most simple groups (and not just the last two categories of examples).

\begin{proof} We use the standard notation for the conjugacy classes of $S_n$: $[1^{\ell_1} \cdots  n^{\ell_n}]$ is the conjugacy class of elements of $S_n$ that write as a product of $\ell_1$ cycles of length $1$, ..., $\ell_n$ cycles of length $n$, all cycles having disjoint supports.

The symmetric groups $S_n$, $n \geq 5$, satisfy $\nu(G) \geq {\rm rk}(G)+2$. Indeed ${\rm rank}(S_n)=2$
and the $4$ conjugacy classes 
\vskip 1,5mm

\centerline{$[n^1]$, $[(n-1)^1]$, $[(n-2)^12^1]$, $[2^1]$} 
\vskip 1,5mm

\noindent
are pairwise very different.

So do the alternating groups $A_n$ with $n\geq 6$: note that ${\rm rank}(A_n)=2$ and 
use the classes 
\vskip 1,5mm

\centerline{$\left\{\begin{matrix}
& \hbox{$[n^1]$, $[(n-3)^12^1]$, $[(n-2)^11^2]$, $[(n-4)^1 1^4]$ if $n$ odd} \hfill \\
&\hbox{$[(n-1)^1]$, $[(n-2)^12^1]$, $[(n-3)^3 3^1]$, $[(n-5)^1 1^5]$ if $n$ even} \hfill \\
\end{matrix}\right.$
}
\vskip 1,5mm


For the second class of examples, we start with the case $G$ is abelian.
If $G$ of rank $s \geq 2$, it then writes $G=\Zz/d_1\Zz \times \cdots \Zz/d_s\Zz$ with $s\geq 2$ and $d_1 | d_2 | \cdots | d_s$ in $\Zz$. The $s$-tuples $(\varepsilon_1, \ldots, \varepsilon_{s-1},1)$ with $\varepsilon_i \in \Zz/d_1\Zz \times \cdots \Zz/d_{s-1}\Zz$ generate non-conjugate maximal cyclic subgroups of $G$. There are $d_1\cdots d_{s-1}$ such $s$-tuples, and so at least $s+2$ unless ($s=2$ and $d_1\in \{2,3\}$) or ($s=3$ and $d_1=d_2=2$). After checking separately the remaining special cases (use further 
non-conjugate maximal cyclic subgroups {\it e.g.} those generated by $s$-tuples $(\varepsilon_1, \ldots, \varepsilon_{s-1},k)$ with $k\in (\Zz/d_s\Zz)^\times$), conclude that $\nu(G)\geq {\rm rk}(G)+2$ unless $G = \Zz/2\Zz \times \Zz/2\Zz$.

Assume now more generally that $G$ is nilpotent. From the Burnside basis theorem, $G$ and its abelianization $G^{{\rm ab}}$ have the same rank. On the other hand, we have $\nu(G) \geq \nu(G^{{\rm ab}})$. If $G$ is non cyclic then so is $G^{{\rm ab}}$. If $G^{{\rm ab}}$ is  further assumed to be different from $\Zz/2\Zz \times \Zz/2\Zz$, then from the preceding case, we have $\nu(G^{{\rm ab}})\geq {\rm rk}(G^{{\rm ab}})+2$. Inequality $\nu(G)\geq {\rm rk}(G)+2$ follows.

All finite simple groups have rank $2$ and their well-known classification shows that many of them have at least $4$ non-conjugate maximal cyclic subgroups of $G$, including all groups ${\rm PSL}_2(\Ff_p)$ ($p>7$ prime), 
all sporadic simple groups.
\end{proof}

\begin{remark}  Depending on the problem and the situation, the general method can be used differently and leads to variants of corollary \ref{cor:criterion-overC}. Here is an example:
\vskip 1mm

\noindent
(*) {\it If $N$ is the largest integer $< \nu(G)/({\rm rk}(G)+1)$, there do not exist $N$ Galois extensions $F_1/\Cc(T),\ldots,F_N/\Cc(T)$ of group $G$ such that every extension $L/\Cc(U)$ of group $G$ is a specialization $(F_{i})_{T_0}/\Cc(U)$ of some $F_i/\Cc(T)$, $i=1,\ldots,N$ {\rm (}for some $T_0\in \Cc(U)${\rm )}.}
\vskip 1mm

\noindent
{\it Proof}. Let $g_1,\ldots,g_{\nu(G)}$ be generators of $\nu(G)$ non-conjugate maximal cyclic subgroups and let $C_1,\ldots,C_{{\nu(G)}}$ be their conjugacy classes.
For $i=1,\ldots,\nu(G)$, construct a Galois extension $L_i/\Cc(T)$ such that $C_i$ appears in ${\bf C}_{L_i}$, $r_{L_i}={\rm rk}(G)+1$, and in such a way that the constructed extensions are distinct; if two extensions happen to be equal in a first stage, compose one with a non-trivial automorphism of 
$\Cc(T)$. Assume that there exist $N$ Galois extensions $F_1/\Cc(T)$, $\ldots$, $F_N/\Cc(T)$ satisfying the conclusion of the claim. Then there is an index $i\in \{1,\ldots,N\}$ such that $F_i/\Cc(T)$ specializes to at least $\nu(G)/N$ of the constructed extensions $L/\Cc(T)$. If ${\mathcal R}$ is the set of these extensions $L/\Cc(T)$,
we have $\rho_{\mathcal R} = {\rm rk}(G)+1$ and criterion \ref{crit:non-param-gen-criterion} can be applied with 
$\nu_{\mathcal R} \geq \nu(G)/N$; this gives $\nu(G)/N \leq {\rm rk}(G)+1$ and so $N\geq \nu(G)/({\rm rk}(G)+1)$.
\vskip 1mm

\end{remark}

\subsubsection{Groups with no $\Qq(U)$-parametric extension} \label{ssec:examples-non-param}

Here we apply criterion \ref{crit:non-param-gen-criterion} over a non-algebraically closed field $k$. 

\begin{corollary} \label{cor:non-param-overQ}
Let $k$ be a subfield of $\Cc$. \hbox{None of the groups}
\vskip 0,5mm

\noindent
- $S_n$, $n\geq 5$ and $n\not=6$ and $A_n$, $n \geq 7$, 
\vskip 0,5mm

\noindent 
- ${\rm PSL}_2(\Ff_p)$ with $p$ a prime such that $(\frac{2}{p})=(\frac{3}{p})=-1$, 
\vskip 0,5mm

\noindent 
- the Fischer-Griess Monster $M$,

\noindent
have a weakly $k(U)$-parametric extension $F/k(T)$  and a fortiori they do not have a $k(U)$-parametric extension. More precisely conclusion {\rm (**)} from criterion \ref{crit:non-param-gen-criterion} holds with $k=\Qq$ and ${\mathcal R}$ consisting of two extensions except the alternating groups $A_n$ for which three are needed.
\end{corollary}

The list is not exhaustive. This corollary is meant to show on examples how to apply criterion \ref{crit:non-param-gen-criterion} and how in some situations where it cannot be applied directly, one can still get the desired conclusion.


\begin{proof} Take $G=S_n$, $n \geq 5$, $n\not=6$. Assume first $n$ is odd.
There are $\Qq$-regular realizations $L_1/\Qq(T)$, $L_2/\Qq(T)$ of $S_n $ with $r_{L_1}=r_{L_2}=3$ and 
\vskip 2mm

\centerline{${\bf C}_{L_1} = ([n^1], [(n-1)^11^1], [2^11^{n-2}])$ \& ${\bf C}_{L_2} = ([n^1], [(n-2)^12^1], [2^11^{n-2})$}
\vskip 2mm

\noindent
(see \cite{schinzel-book2000}, \cite[B-3]{legrand-thesis}). 
Hence for ${\mathcal R} = \{L_1/\Qq(T), L_2/\Qq(T)\}$, we have $\rho_{\mathcal R} \leq 3$ and one can take
$\nu_{\mathcal R} \geq 4$ in criterion \ref{crit:non-param-gen-criterion}. Conclude that (**) is satisfied for this ${\mathcal R}$. In particular, $S_n$ has no weakly $k(U)$-parametric extension $F/k(T)$.  The case $n$ is even is similar; $L_2/\Qq(T)$ should be changed to have ${\bf C}_{L_2} = ([n^1], [(n-m)^1m^1], [2^11^{n-2})$ for some integer $m$ such that $1\leq m\leq n$ and $(m,n)=1$.

Take $G=A_n$, $n \geq 7$. The group $A_n$ is known to have $\Qq$-regular realizations with the following
inertia canonical invariant \cite[B-3]{legrand-thesis}:

\vskip 1mm
\noindent
if $n$ is even:
\vskip 0,5mm

\centerline{$([m^1(n-m)^1],[m^1(n-m)^1],[(n/2)^2])$ with $1\leq m\leq n$, $(m,n)=1$}

\vskip 1mm
\noindent
if $n$ is odd:
\vskip 1mm

\centerline{$\begin{matrix}
&\hbox{$([n^1],[n^1],[m^1((n-m)/2)^2])$ with $m$ odd, $1\leq m\leq n$, $(m,n)=1$,} \hfill \\
&\hbox{$ ([n^1],[n^1],[(m/2)^2(n-m)^1])$ with $m$ even, $1\leq m\leq n$, $(m,n)=1$} \hfill \\
\end{matrix}$
}
\vskip 1,5mm

%
%

\noindent
One checks that for every $n\geq 7$, one can always find three such realizations with four pairwise very different conjugacy classes in the union of the three inertia canonical invariants.
Criterion \ref{crit:non-param-gen-criterion} concludes that the proof in this case.

Take $G={\rm PSL}_2(\Ff_p)$ with $p$ a prime such that $(\frac{2}{p})=-1$ and $(\frac{3}{p})=-1$.
\cite[\S 8.3.3]{Serre-topics} gives two $\Qq$-regular realizations $L_1/\Qq(T)$ and $L_2/\Qq(T)$ of $G$ with $r_{L_1}=r_{L_2}=3$ and 
\vskip 2mm

\centerline{${\bf C}_{L_1} =(2A,pA,pB)$ \& ${\bf C}_{L_2} = (3A,pA,pB)$}
\vskip 2mm

\noindent
where $2A$ (resp. $3A$) is the unique conjugacy class of ${\rm PSL}_2(\Ff_p)$ of order $2$ (resp. of order $3$) and $pA, pB$ are the two conjugacy classes of order $p$.
Hence for ${\mathcal R} = \{L_1/\Qq(T), L_2/\Qq(T)\}$, we have $\rho_{\mathcal R} \leq 3$.

According to \cite[corollary 2.7]{guest-at-al}, 
the maximal order of an element of ${\rm PSL}_2(\Ff_p)$ is $p+1$, so the conjugacy classes $pA$ and $pB$ are classes of generators of maximal cyclic subgroups. It follows that $2A \verydiff pA$, 
$2A \verydiff pB$, $3A \verydiff pA$, $3A \verydiff pB$. Furthermore $2A\verydiff 3A$: indeed otherwise both classes would be contained in the conjugacy class of a cyclic subgroup $\langle \gamma_0 \rangle\subset {\rm PSL}_2(\Ff_p)$ of order $6$. But then $\gamma_0 \in 3A$ or $ (-\gamma_0)\in 3A$ and so $2A \subset (3A)^\Zz$ or $2A \subset (-3A)^\Zz$ -- a contradiction.

However $pA$ and $pB$ are not very different and criterion \ref{crit:non-param-gen-criterion} cannot be applied directly. We use instead the following argument. Assume there is a $k$-regular extension $F/k(T)$ such that  
$F\Cc/\Cc(T)$ specializes to $L_1\Cc/\Cc(T)$ and $L_2\Cc/\Cc(T)$.
From above we have $r_F= 3$ and, for ${\bf C}_F=(C_1,C_2,C_3)$, there should exist integers $a_i,b_i,c_i \in \Zz$, $i=1,2$, such that 
$(C_1^{a_1},C_2^{b_1},C_3^{c_1}) = (2A,pA,pB)$ and $(C_1^{a_2},C_2^{b_2},C_3^{c_2}) = (3A,pA,pB)$.
Necessarily $C_2,C_3 \in \{pA,pB\}$, $2A \subset C_1^\Zz$ and $3A \subset C_1^\Zz$. This contradicts $2A\verydiff 3A$.

Finally take $G=M$ the Fischer-Griess Monster. We will use two known $\Qq$-regular realizations $L_1/\Qq(T)$ and 
$L_2/\Qq(T)$ of $G$, for which $r_{L_1}=r_{L_2}=3$ and 
\vskip 1mm

\centerline{${\bf C}_{L_1} =(2A,3B,29A)$  and  ${\bf C}_{L_2} = (2A,3C,38A)$}
\vskip 1mm

\noindent
(where we use the standard notation from the Atlas of simple groups for the conjugacy classes of $M$). The extension $L_1/\Qq(T)$ is the one originally produced by J. Thompson \cite{thompson}; the main point is that  
${\bf C}_{L_1}$ is a ``rigid triple''. Computer programs now exist to find other rigid tuples. The triple ${\bf C}_{L_2}$ was communicated to me by J. Koenig who checked that it is rigid, assuming that the current classification of all (certain and hypothetical) maximal subgroups of $M$ is correct.


Assume there is a $k$-regular extension $F/k(T)$ such that  
$F\Cc/\Cc(T)$ specializes to $L_1\Cc/\Cc(T)$ and $L_2\Cc/\Cc(T)$. Then we have $r_F= 3$. Set ${\bf C}_F=(C_1,C_2,C_3)$. From the Atlas of simple groups, there is only one conjugacy class, $38A$, whose elements are of order a multiple of $38$ (and this multiple is $38$) and there are three conjugacy classes, $29A$, $87A$ and $87B$, whose elements are of order a multiple of $29$ (and these multiples are $29$, $87$ and $87$). One of $C_1,C_2,C_3$, say $C_1$, must be $38A$ and one, say $C_2$, should be $29A$ or $87A$ or $87B$. Furthermore $3B$ and $3C$ are not a power of $87A$ or $87B$. This leads to these possibilities for the triple ${\bf C}$ of ramification indices of $F/\Qq(T)$: 
\vskip 1mm

\centerline{${\bf C} = (38A,29A,C_3)$ or ${\bf C} = (38A,87A,C_3)$ or ${\bf C} = (38A,87B,C_3)$}
\vskip 1mm

\noindent
with $C_3$ of order divisible by $3$.
But then the lower bound for the number $r_{T_0}$ of branch points of a specialization $F_{T_0}/\Qq(T)$ with $T_0\in \Qq(T)$ given in theorem \ref{thm:r-in-spec} (b-1) 
gives $r_{T_0} >3$ 
and so neither $L_1/\Qq(T)$ nor $L_2/\Qq(T)$ can be a specialization of an extension $F/\Qq(T)$ with inertia canonical invariant ${\bf C}$.
\end{proof}

\subsection{Non $\Qq$-parametric extensions $F/\Qq(T)$} \label{ssec:non-Q-param-extensions} 
Assume that $k$ is a number field. 

\subsubsection{Main tool for producing non $k$-parametric extensions $F/k(T)$}

\begin{theorem} \label{thm:twist-main} Let $F/k(T)$ and $L/k(T)$ be two $k$-regular Galois extensions with respective groups $G$ and $H$ such that $H\subset G$. There exist polynomials 
$\widetilde P_1,\ldots, \widetilde P_N \in k[U,T,Y]$
with the following properties:
\vskip 0,8mm

\noindent
{\rm (a)} $\widetilde P_i$ is monic in $Y$, $\deg_Y(\widetilde P_i) = |G|$ and the splitting field over $\overline{k(U)}(T)$ of $\widetilde P_i$ is the extension $F\overline{k(U)}/\overline{k(U)}(T)$, $i=1,\ldots,N$,
\vskip 1,3mm

%
%

\noindent
{\rm (b)} For every field $K$ with $k\subset K \subset \Cc(U)$, for all $u_0\in K$ but finitely many in $\overline k$, and for all $t_0\in K$ not a branch point of $F/k(T)$, the specialization $L_{u_0}/K$ is $K$-isomorphic to the specialization $F_{t_0}/K$ iff for some $i\in \{1,\ldots,N\}$, there exists $y_0\in K$ such that $\widetilde P_i(u_0,t_0,y_0)=0$.
\end{theorem}

Theorem \ref{thm:twist-main} is proved in \S \ref{sec:twisting}.


\subsubsection{The working hypothesis} \label{ssec:working-hypothesis}
We will deduce some non $k$-parametric conclusions from theorem \ref{thm:twist-main} under this
 ``working hypothesis'':

\vskip 1,5mm

\noindent
(WH) {\it Given a number field $k$ 
and polynomials $P_1,\ldots,P_N \in k[U,T,Y]$ irredu\-cible in $\overline{k}[U,T,Y]$, we have the following: if none of the equations \hbox{$P_i(U,t,y)=0$} has a solution $(T_0,Y_0)\in \Cc(U)^2$, then for infinitely many $u_0\in k$, none of the equations $P_i(u_0,t,y)=0$ has a solution
$(t_0,y_0)\in k^2$ {\rm (}$i=1,\ldots,N${\rm )}.}
\vskip 1,5mm

We comment on this hypothesis below in \S \ref{ssec:comments-WH}.

\subsubsection{Main conclusions} 
We first explain how we combine our working hypothesis and theorem \ref{thm:twist-main}.


\begin{proposition} \label{cor:Q-param-extensions}
Assume {\rm (WH)} holds and let $k$ be a number field 
and $F/k(T)$ be a $k$-regular Galois extension of group $G$.
\vskip 0,5mm

\noindent
{\rm (a)} If a $k$-regular Galois extension $L/k(U)$ 
of group $H\subset G$ is such that $L\Cc/\Cc(U)$ is not a specialization of $F/k(T)$ 
 at any $T_0\in \Cc(U)$, 
there are infinitely many $u_0\in k$ such that the specialization 
$L_{u_0}/k$ of $L/k(U)$ is not a specialization of $F/k(T)$ at any unbranched $t_0\in k$.
\vskip 1mm

\noindent
{\rm (b)} If $F/k(T)$ is $k$-parametric then it is weakly $k(U)$-parametric.

\vskip 1mm

\noindent
{\rm (c)} Every group with no weakly $k(U)$-parametric extension has no $k$-parametric extension.
\end{proposition}


\begin{corollary} \label{cor:Q-param-extensions2}
Assume {\rm (WH)} holds and let $k$ be a number field. Every group as in  corollary \ref{cor:non-param-overQ}
has no \hbox{$k$-parametric extension.}  
\end{corollary}

\begin{proof}[Proof of proposition \ref{cor:Q-param-extensions}]
%
Let $F/k(T)$ and $L/k(U)$ be as in (a) and $\widetilde P_1,\ldots, \widetilde P_N$ be the polynomials provided by theorem \ref{thm:twist-main}.  From this result and the assumption in (a), none of
the equations $\widetilde P_i(U,t,y)=0$ has a solution $(T_0,Y_0)\in \Cc(U)^2$, $i=1,\ldots,N$. Under {\rm (WH)}, one may conclude that for infinitely many $u_0\in k$, none of
the equations $\widetilde P_i(u_0,t,y)=0$ has a solution $(t_0,y_0)\in k^2$. From theorem \ref{thm:twist-main}, for these $u_0$, the specialization $L_{u_0}/k$, which is of Galois group contained in $G$, is not 
a specialization $F_{t_0}/k$ of $F/k(T)$ with $t_0\in k$. Statements (b) and (c) follow straightforwardly from (a).
\end{proof}

\begin{remark} \label{rem:on-corollary-with-WH} (a) The proof shows that proposition \ref{cor:Q-param-extensions} and corollary \ref{cor:Q-param-extensions2} still hold if (WH) is replaced by the weaker hypothesis (WH-Gal/$\Cc$) for which the conclusion of (WH)  solely holds for polynomials $P_1,\ldots,P_n$ such that

\noindent
-  $P_i(U,T,Y)$ is irreducible in $\overline{k(U)}[T,Y]$, $i=1,\ldots,N$,

\noindent
- the splitting field over $\overline{k(U)}(T)$ of $\widetilde P_i$ is the extension $F\overline{k(U)}/\overline{k(U)}(T)$, that is, this splitting field  is $\overline{k(U)}(T)$-isomorphic to $F\overline{k(U)}$, $i=1,\ldots,N$.

%
%
%

\vskip 1,5mm

%

\noindent
{\rm (b)} We can now explain how, conditionally, Legrand's result (mentioned in \S \ref{sec:intro}.1) can be deduced from ours. Assuming $G$ is the group of some $\Qq$-regular Galois extension $L/\Qq(U)$, if $F/\Qq(T)$ is another $\Qq$-regular Galois extension of group $G$ such that $L\Cc/\Cc(U)$ is not a $\Cc(U)$-specialization of $F/\Qq(T)$, then it follows from proposition \ref{cor:Q-param-extensions} that, if $G$ satisfies (WH-Gal/$\Cc$), $F/\Qq(T)$ is not $\Qq$-parametric.

For example, one can take for $F/\Qq(T)$ a specialization $L_{U_0}/\Qq(T)$ with $U_0(T)= a + T^5$ with $a\in \Qq$. For all but finitely many $a\in \Qq$, $\Gal(F/\Qq(T)) = G$. From inequality  (*) from \S \ref{ssec:comparizon}, the branch point number of  $F\Cc/\Cc(T)$ 
is bigger than that of $L\Cc/\Cc(U)$. From theorem \ref{thm:order-invariants}, the latter is not a specialization of the former with $T_0\in \Cc(U)$.

\end{remark}

\subsubsection{Comments on the working hypothesis} \label{ssec:comments-WH}

\hskip 1mm
\vskip 0,5mm

\noindent
(a) The working hypothesis (WH) is known to hold in these situations:
\vskip 0,7mm

\noindent
- $\deg_T(P)=0$: (WH) is then Hilbert's Irreducibility Theorem; (WH) is an extension of HIT to $2$-indeterminate polynomials.
\vskip 1mm

\noindent
- the affine $\overline{k(U)}$-curve of equation $P(U,t,y)=0$ is of genus $0$. This follows from Schinzel's thm 38 in \cite{schinzel-book}. A stronger conclusion even holds: the infinitely many $u_0$ can be chosen in the ring of integers of $k$.
\vskip 0,5mm

\noindent
- the polynomial $P(U,T,Y)$ is of the form $Y^n-U^m Q(T)$ with $n\geq 2$ dividing $\deg(Q)$, $m$ relatively prime to $n$ and $Q\in k[T]\setminus k$ a polynomial with integral coefficients such that the Galois group of $Q$ has an element that fixes no root of $Q$ (\hbox{\it e.g.} $Q$ is irreducible in $k[T]$) \cite{legrand-hyper-elliptic}.
\vskip 1mm

\noindent
(b) There is no known counter-example to (WH). 
In this context, Cassels and Schinzel \cite{cassels-schinzel} consider the family of genus $1$ curves $y^2= t(t^2-(7+7U^4)^2)$ and show that this equation has no solution $(T_0(U),Y_0(U) \in \Qq(U)^2$ and that, under a conjecture of Selmer  \cite{selmer}, for every $u_0\in \Qq$, the equation $y^2= t(t^2-(7+7u_0^4)^2)$ has a solution $(t_0,y_0)\in \Qq^2$.
This however does not provide a counter-example (even conjectural) to (WH) as the equation $y^2= t(t^2-(7+7U^4)^2)$ also has solutions in $\Cc(U)$.

\vskip 1mm

\noindent
{\rm (c)} While it may be hard to find a non $k$-parametric extension $F/k(T)$ to test the working hypothesis, it is possible to produce extensions $L/k(U)$ and $F/k(T)$ with the former not a specialization of the latter (for example do as above in (b)). The polynomials $\widetilde P_1,\ldots,\widetilde P_N$ given by theorem \ref{thm:twist-main} then are good candidates to test  the working hypothesis.
\vskip 1mm

\section{$\Cc(U)$-specializations of Galois extensions $F/\Cc(T)$} \label{sec:C(U)-specializations} 


Let $F/\Cc(T)$ be a degree $d$ Galois extension of group $G$, with $r$ branch points $t_1,\ldots,t_r$, inertia canonical invariant ${\bf C} = (C_1,\ldots,C_r)$ and associated ramification indices ${\mathbf e} = (e_1,\ldots,e_r)$. Also set
\vskip 2,5mm

\centerline{$\displaystyle \left\{\begin{matrix}
& \displaystyle \varepsilon =  \frac{1}{e_1} + \cdots +\frac{1}{e_r}\hfill \cr
& e_\infty = \max(e_1,\ldots,e_r) \hfill\cr
\end{matrix}\right.$}
\vskip 2mm

Let $T_0(U)=a(U)/b(U) \in \Cc(U)\setminus \Cc$ with $a,b\in \Cc[U]$ relatively prime and $b\not=0$.
 Set 
\vskip 1mm
\centerline{$N = \deg(T_0)=\max(\deg(a),\deg(b))$} 
\vskip 2mm

We will compare the invariants of $F/\Cc(T)$ to those of $F_{T_0}/\Cc(T)$.

Note that when $N=1$, $T_0$ is a linear fractional transformation and the two extensions $F/\Cc(T)$ and $F_{T_0}/\Cc(T)$ are isomorphic. More specifically, $T_0$ interprets as an automorphism of $\Pp^1(\Cc)$ and if $f:X\rightarrow \Pp^1$ is the branched cover corresponding to  $F/\Cc(T)$, then $F_{T_0}/\Cc(T)$ corresponds to the cover $f\circ T_0^{-1}$. In particular the invariants $G$, $r$, $d$, $g$, ${\bf C}$ are the same for the two extensions.

\subsection{Invariants of the specialized extensions} \label{ssec:invariants}
Denote the invariants of the specialized extensions $F_{T_0}/\Cc(U)$ by $G_{T_0}$, $d_{T_0}=|G_{T_0}|$, $g_{T_0}$ and ${\bf C}_{T_0}$. The following statement is the most precise of this section. We will in particular deduce theorem \ref{thm:order-invariants} from it.

\begin{theorem} \label{thm:r-in-spec}  Consider the specialized extension $F_{T_0}/\Cc(U)$.
\vskip 1mm

\noindent
{\rm (a)} We have $G_{T_0}\subset G$, equivalently $d_{T_0}\leq d$. Furthermore $d_{T_0}<d$ if and only if there is a subfield $L\subset F$, $L\not=\Cc(T)$, of genus $0$, such that a specialization of it at $T_0$ is trivial: $L_{T_0}=\Cc(U)$.

\vskip 1,5mm

\noindent
{\rm (b)} The  branch point number $r_{T_0}$ 
satisfies $r_{T_0} \leq rN$ and

\vskip 1,5mm

{\rm (b-1)} \hskip 20mm $\displaystyle r_{T_0}  \geq \frac{(r-\varepsilon - 2) N +2}{1-(1/e_\infty)}$\hskip 5mm if $r\geq 0$

\vskip 2mm

{\rm (b-2)} \hskip 20mm $\displaystyle  r_{T_0}  \geq (r-4)N+4 $  \hskip 11mm if $r\geq 4$

%
%

\vskip 3mm

\noindent
{\rm (c)} The inertia canonical invariant ${\bf C}_{T_0}$ of $F_{T_0}/\Cc(U)$ consists of conjugacy classes in $G_{T_0}$ of powers $g^\alpha$ {\rm (}$\alpha\in \Nn$ {\rm )} of elements of $C_1 \cup \ldots \cup C_r$.
\vskip 2mm

\noindent
{\rm (d)} The genus $g_{T_0}$ satisfies  $g_{T_0} \leq N(g+d-1)$, and, if $G_{T_0}=G$, 
\vskip 0,5mm
\centerline{$\displaystyle g_{T_0}-g\geq  \frac{d}{4} (N-1)(r-4)$}
\end{theorem}

\begin{remark} The lower and upper bounds for $g_{T_0}$ in (d) are better than those that can be deduced from inequalities (b-1) or (b-2) by combining them with the usual ones given by Riemann-Hurwitz:
\vskip 1mm

\centerline{$\displaystyle \frac{r}{2} + 1 -n \leq g \leq \frac{rn}{2} + 1 -n -\frac{r}{2}$}
\vskip 1,5mm

%
%
%
%

\end{remark}

\begin{proof} (a) The first part of (a) is standard. 

Assume that there is a subfield $L\subset F$, $L\not=\Cc(T)$ with a trivial specialization: $L_{T_0}=\Cc(U)$. Then we have 
\vskip 1mm

\centerline{$d_{T_0} = [F_{T_0}:L_{T_0}] \hskip 2pt [L_{T_0}:\Cc(U)]  \leq [F:L] < d$.} 
\vskip 1mm

For the converse, assume that $d_{T_0}<d$. A standard argument ({\it e.g.} \cite[lemma 13.1.2]{FrJa}) from the theory of hilbertian fields (applied here to the field $\Cc(U)$) shows that there exists  $\theta \in F\setminus \Cc(T)$ such that $\Cc(T,\theta)_{T_0}=\Cc(U)$: if $P(T,Y)$ is an affine equation of $F/\Cc(T)$, $\theta$ is a coefficient in $\overline{\Cc(T)}$ of a factorisation $P(T,Y)$ in $\overline{\Cc(T)}[Y]$. The field $L=\Cc(T,\theta)$ 
is the desired field. 

That $L$ is of genus $0$ follows from
$L_{T_0}=\Cc(U)$. Indeed, if $Q(T,Y)$ is an affine equation for $L/\Cc(T)$, $L_{T_0}=\Cc(U)$ means that
there exists $Y_0(U)\in \Cc(U)$ such that $Q(T_0(U),Y_0(U))=0$, which is a rational parametrization 
of the curve of equation $Q(T,Y)$. Hence its function field $L$ is of genus $0$.
\vskip 1,5mm


\noindent
(b) A first point of the proof is that
\vskip 1,5mm

\noindent
{\rm (*)} {\it if $u\in \Pp^1(\Cc)$ is a branch point of $F_{T_0}/\Cc(U)$, there exists a branch point $t_i$ of $F/\Cc(T)$ such that $T_0(u) = t_i$ and, conversely, if $T_0(u) = t_i$, the associated inertia group is generated by some power $g_i^\alpha$ {\rm (}$\alpha\in \Nn$ {\rm )} of an element $g_i \in C_i$}.
\vskip 1,5mm

\noindent  
This statement, which in particular yields conclusion (c), follows from the Specialization Inertia Theorem (SIT)  recalled in \S \ref{sec:preliminaries}. More specifically, we use it in the situation that the Dedekind domain is $A=\Cc[U]$ (or $A=\Cc[1/U]$ for $u=\infty$), the $K$-regular extension  is $F\Cc(U)/\Cc(U)(T)$ and ${\frak p}$ is the ideal ${\frak p} = \langle U-u \rangle$ if $u \in \Cc$ (and ${\frak p} = \langle 1/U \rangle$ if $u=\infty$). 

A few remarks on the assumptions from \S \ref{sec:preliminaries}  are in order:
\vskip 0,5mm

\noindent
(1) $|G|\notin {\frak p}$ since $A/{\frak p}=\Cc$ is of characteristic $0$.
\vskip 0,5mm

\noindent
(2)  {\it there is no vertical ramification at ${\frak p}$ in the extension $FK/K(T)$:}
indeed if ${\mathcal Y}$ is a primitive element of $F/\Cc(T)$, integral over $A$, then $1,{\mathcal Y}, \ldots, {\mathcal Y}^{d-1}$ (with $d=|G|$) are integral over $A$ and over $A[T]$ as well, and form a $K(T)$-basis of $F\Cc(U)/\Cc(U)(T)$. The discriminant of this basis is a non-zero element of $\Cc\subset A$, and so remains non-zero modulo ${\frak p}$. This classically guarantees the content of our claim.
\vskip 0,5mm

\noindent
(3) {\it no two different branch points of $F/K(T)$ meet modulo ${\frak p}$}: indeed the branch points are those of $F/\Cc(T)$ and their mutual differences $t_i-t_j$ or $(1/t_i)-(1/t_j)$ are non-zero elements of $\Cc\subset A$ and remain non-zero modulo ${\frak p}$.
\vskip 0,5mm

\noindent
(4) {\it the ideal ${\frak p}$ is unramified in $K(t_1,\ldots,t_r)/K=\Cc(U)/\Cc(U)$.}
\vskip 0,5mm

\noindent
(5) {\it $t_i$ and $1/t_i$ are integral over $\widetilde A_{\frak p}$}: $t_i, 1/t_i\in \Cc \cup \{\infty\}$ $i=1,\ldots,r$.
\vskip 0,5mm

%

The SIT concludes that if $u\in \Pp^1(\Cc)$ is a branch point of $F_{T_0}/\Cc(U)$, there exists $i\in \{1,\ldots,r\}$ such that $T_0$ meets $t_i$ modulo ${\frak p}$, which exactly means that $T_0(u)=t_i$, and, secondly, that, if $T_0(u)=t_i$, the inertia group of $F_{T_0}/\Cc(U)$ at $u$, is generated by an element of $C_i^\alpha$ 
with 
\vskip 2mm

\noindent
(**) \hskip 20mm $\alpha= {\rm ord}_u(T_0(U)-t_i)$
\vskip 2mm

This concludes the proof of (*). To simplify the exposition of the rest of the proof, we first assume:
\vskip 1,5mm

\noindent
(H) {\it neither $\infty$ nor $T_0(\infty)$ is a branch point of $F/\Cc(T)$}. 
\vskip 1,5mm


\noindent
and will explain afterwards how to reduce to this hypothesis. 

For an element $u\in \Pp^1(\Cc)$ such that $T_0(u) = t_i$ for some $i=1,\ldots,r$ to be a branch point of $F_{T_0}/\Cc(U)$, the integer $\alpha$ from (**) should not be a multiple of $e_i$. Note further that because of (H), $u\not=\infty$ and $u$ is not a pole of $T_0$. 

For each $i=1, \ldots,r$, label the roots in $\Cc$ of $a(U)-t_ib(U)$ as follows:

\noindent
$\sxbullet$ $u_{i1},\ldots, u_{ip_i}$ are the $p_i$ distinct simple roots, 

\noindent
$\sxbullet$ $v_{i1},\ldots, v_{iq_i}$ are the $q_i$ distinct multiple roots of orders, say $m_{i1},\ldots,$ $m_{iq_i}$, non divisible by $e_i$,

\noindent
$\sxbullet$ $w_{i1},\ldots, v_{is_i}$ are the $s_i$ distinct multiple roots of orders, say $n_{i1},\ldots,$$ n_{iq_i}$, divisible by $e_i$.
\vskip 0,5mm

\noindent
Then we have
\vskip 1,5mm

\noindent
(1) $\displaystyle p_i + \sum_{j=1}^{q_i} m_{ij} + \sum_{j=1}^{s_i} n_{ij} = N$, $i=1,\ldots,r$,
\vskip 1,5mm

\noindent
(2) $\displaystyle\sum_{i=1}^r  \left(\sum_{j=1}^{q_i} (m_{ij}-1) + \sum_{j=1}^{s_i} (n_{ij} -1)\right) \leq 2N-2$.
\vskip 1,5mm

Equality (1) is clear. As to (2), it follows from the fact (left to the reader\footnote{with this hint:  if $a^{(k)}\hskip -2pt (u) - t_i b^{(k)}\hskip -2pt (u) = 0$ for  $k=0,\ldots,m-1$, then  $a^{(h)}\hskip -2pt (u) b^{(k)}\hskip -2pt (u)- a^{(k)}\hskip -2pt (u)  b^{(h)}\hskip -2pt (u) = 0$ for $h,k=0,\ldots,m-1$. The claim follows from the observation that every derivative $(a^\prime b - a b^\prime)^{(h)}$ ($h=0,\ldots,m-2$) is a sum of terms of the form $a^{(h)} b^{(k)}- a^{(k)}  b^{(h)}$ with $h,k=0,\ldots,m-1$. }) that if $u\in \Cc$ is root of $a(U)-t_ib(U)$ of order $m\geq 1$ for some $i=1,\ldots,r$, then $u$ is a root of order $m-1$ of the polynomial $a^\prime b  - a b^\prime $, which is of degree $2N-2$. An alternate argument consists in applying the Riemann-Hurwitz formula to the branched cover $T_0:\Pp^1\rightarrow \Pp^1$ induced by the rational function $T_0(U)$: the left-hand side term from (2) is smaller than or equal to the term $\sum_P (e_P-1)$ of this formula (where $P$ ranges over all ramified points) and so is $\leq 2\cdot 0 - 2 + 2 \deg(T_0) = 2N-2$.

Statement (*) gives 

\centerline{$\displaystyle r_{T_0} = \sum_{i=1}^r (p_i+q_i)$}
\vskip 0,5mm

\noindent
Clearly $r_{T_0} \leq rN$ follows. Inequality (2), conjoined with (1), rewrites
\vskip 1mm

\centerline{$\displaystyle\sum_{i=1}^r  (N-p_i - q_i -s_i) \leq 2N-2$}
\vskip 1mm

\noindent
so we obtain
\vskip 1,5mm

\noindent
(***) \hskip 25mm $\displaystyle r_{T_0} \geq (r-2)N + 2 - \sum_{i=1}^r s_i$
\vskip 1,5mm

\noindent
From (1), for $i=1,\ldots,r$, we also have $p_i+q_i + e_i \hskip 1pt s_i \leq N$, whence 
\vskip 1,5mm

\centerline{$\displaystyle s_i \leq \frac{N}{e_i} - \frac{p_i+q_i}{e_i}$}
\vskip 1,5mm

\noindent
The definition of $e_\infty$ and $\varepsilon$ leads to
\vskip 1,5mm

\centerline{$\displaystyle r_{T_0}  \geq rN - (2N-2) - \varepsilon N + \frac{1}{e_\infty}  \sum_{i=1}^r (p_i+q_i) $}

\vskip 1,5mm
\noindent
and finally to inequality (b-1).
%
%

From (2) we can also deduce that 
\vskip 1mm

\noindent
\centerline{$\displaystyle \sum_{i=1}^r s_i \leq 2N-2$}
\vskip 1mm

\noindent
 which conjoined with (***), yields inequality (b-2). 
 
%
%
%
%
 \vskip 1,5mm 
 
\noindent
(d) Write the Riemann-Hurwitz formula for $F/\Cc(T)$ and $F_{T_0}/\Cc(U)$:
 \vskip 2,5mm
%
%
%
   
    \centerline{$ \left\{ \begin{matrix}
    &\displaystyle 2g-2 = -2d +\sum_F (e_{\mathcal P}-1) \cr
&\displaystyle 2g_{T_0}-2 = -2d_{T_0} +\sum_{F_{T_0}} (e_{\mathcal P}-1)\cr
\end{matrix} \right.$}
\vskip 1,5mm

\noindent
where $\sum_{F}$ (resp. $\sum_{F_{T_0}}$) means that the sum ranges over all places of $F$ (resp. of $F_{T_0}$)
 trivial on $\Cc$. The first claim from (d) comes from
 \vskip 2,5mm
 
\centerline{$\displaystyle g_{T_0}= 1-d_{T_0} +\frac{1}{2}\sum_{F_{T_0}} (e_{\mathcal P}-1)\leq \frac{N}{2} \sum_{F} (e_{\mathcal P}-1) = N(g-1+d)$}
 \vskip 1,5mm

 As $F/\Cc(T)$ is Galois, we also have 
 
 \vskip 2mm

\centerline{$\displaystyle \sum_{F} (e_{\mathcal P}-1) = \sum_{i=1}^r \sum_{{\mathcal P}/t_i} (e_i-1) = \sum_{i=1}^r  \left(d-\frac{d}{e_i}\right) $}

\vskip 2mm

\noindent Assume $G_{T_0}=G$, so $d_{T_0}=d$. Our analysis of the branch 
 points of the specialized extension $F_{T_0}/\Cc(U)$ yields:
 \vskip 1mm
 

%
%
%

\centerline{$\displaystyle \sum_{F_{T_0}} (e_{\mathcal P}-1) \geq \ \sum_{i=1}^r  \left(d-\frac{d}{e_i}\right)  p_i$}

 \vskip 1mm
 \noindent
whence

\centerline{$\displaystyle g_{T_0}-g \geq \frac{1}{2} \hskip 2pt \sum_{i=1}^r  \left(d-\frac{d}{e_i}\right) (p_i-1)$}

\vskip 1mm
\noindent
Now we have, for each $i=1,\ldots,r$,

\centerline{$\begin{matrix}
\hfill p_i & \displaystyle = N - \sum_{j=1}^{q_i} m_{ij} - \sum_{j=1}^{s_i} n_{ij} \hfill \\
&\displaystyle \geq N - 2 \left(\sum_{j=1}^{q_i} (m_{ij}-1)+  \sum_{j=1}^{s_i} (n_{ij}-1)\right) \hfill \\
\end{matrix}$}
%
%
%

\noindent
We deduce:
\vskip 2mm
\centerline{$\begin{matrix}
\hfill  g_{T_0}-g \hskip -2mm&\displaystyle \geq \frac{1}{2}\hskip 1mm \hbox{$\displaystyle \sum_{i=1}^r  (d-\frac{d}{e_i})$} \left(N-1 - 2 \hskip 1mm (\sum_{j=1}^{q_i} (m_{ij}-1)+  \sum_{j=1}^{s_i} (n_{ij}-1))\right)\hfill \\
&\displaystyle \geq \frac{d}{4} \left(r(N-1) -2(2N-2)\right) \hfill \\
& \displaystyle = \frac{d}{4} (N-1)(r-4) \hfill \\
\end{matrix}$}
\vskip 4mm

Finally we explain how to reduce to a situation for which assumption (H) is satisfied. Note first that the parameters $r$, $d$, $g$, ${\bf C}$ are unchanged if the extension $F/\Cc(T)$ is replaced by any extension $F_\chi/\Cc(T)$ with $\chi\in \Cc(T)$ of degree $1$.


For some fixed $\theta_0\in \Cc \setminus\{t_1,\ldots,t_r\}$, consider the linear 
fractional transformation $\chi$ defined by
\vskip 1mm

\centerline{$\displaystyle \chi^{-1}(T) = \frac{\tau T+\mu}{T-\theta_0}$}
\vskip 1mm

\noindent
where $\tau, \mu$ are chosen in $\Cc$ so that the complex numbers $\chi^{-1}(t_1)$,$\ldots$, $\chi^{-1}(t_r)$ are different from $\infty$; such a choice is possible as $\Cc$ is infinite. These $r$ complex numbers are the branch points of the extension $F_\chi/\Cc(T)$, and so these branch points are different from $\infty$. Fix then a second linear fractional transformation $\chi^\prime$ such that $T_0(\chi^\prime(\infty)) \notin \{t_1,\ldots,t_r\}$.
By construction the extension $F_\chi/\Cc(T)$ and the rational function $\chi^{-1} \circ T_0 \circ \chi^\prime$ satisfy the assumption (H). Therefore the conclusions from theorem \ref{thm:r-in-spec} comparing the ramification invariants of the specialized extension 
\vskip 1mm

\centerline{$(F_\chi)_{\chi^{-1} \circ T_0 \circ \chi^\prime}/\Cc(U) = F_{T_0\circ \chi^\prime}/\Cc(U)$}
\vskip 1mm

\noindent
with those of $F_\chi/\Cc(T)$ are satisfied. These conclusions hold as well for the invariants of the specialized extension $F_{T_0}/\Cc(U)$ compared to those of $F/\Cc(T)$ since $F_{T_0}/\Cc(U)$ (\hbox{resp.} $F/\Cc(T)$) is obtained from $F_{T_0\circ \chi^\prime}/\Cc(U)$ (\hbox{resp.} $F_\chi/\Cc(T)$) by composition with an automorphism of $\Cc(U)$ (\hbox{resp.}  an automorphism of $\Cc(T)$). 
\end{proof}

In the next subsections, we explain how to theorem \ref{thm:order-invariants} can be deduced from theorem \ref{thm:r-in-spec}.


\subsection{Proof of theorem \ref{thm:order-invariants} \hskip -2mm (a)} \label{sec:branch-point-number-growth}
Assume $g\geq 1$ 
and let $L/\Cc(T)$ be a Galois extension such that $F/\Cc(T) \prec L/\Cc(T)$, \hbox{\it i.e.}, there exists $T_0\in \Cc(U)\setminus \Cc$ such that $L/\Cc(T)= F_{T_0}/\Cc(T)$. As in \S \ref{ssec:invariants} denote the invariants of $F_{T_0}/\Cc(T)$ by $G_{T_0}$, $r_{T_0}$, $g_{T_0}$, ${\bf C}_{T_0}$.

We already know that $G\supset G_{T_0}$ and ${\bf C} \prec {\bf C}_{T_0}$ (theorem \ref{thm:r-in-spec} (a),\hskip -1pt (c)). 

%
%
%

Next we show that $r_{T_0} \geq r$. We may assume that $N\geq 2$.


A first case is when $r\geq 4$: $r_{T_0} \geq r$ follows from theorem \ref{thm:r-in-spec} (b-2). 
\vskip 0,5mm


%
%
%
%
%

From theorem \ref{thm:r-in-spec} (b-1),  if $\varepsilon \leq (r-1)/2$ and $r\geq 3$ we have:

\vskip 1,5mm

\centerline{$\displaystyle  r_{T_0} > (r-\frac{r-1}{2} - 2) N +2 \hskip 1mm \geq  \hskip 1mm 2\hskip -2pt \left(\frac{r}{2} - \frac{3}{2}\right) +2 = r-1$}

\vskip 2mm
\noindent
In particular, for $r=3$, we have $r_{T_0}\geq r$ if $\varepsilon \leq 1$.  A simple check shows that the following $3$-tuples ${\bf e}$:
\vskip 1mm

\centerline{$(2,2,2), (2,3,3), (2,3,4), (2,3,5), (2,2,e)$, $(e\geq 3)$}
\vskip 1mm

\noindent
are exactly those for which $\varepsilon > 1$ and that $g=0$ in these cases. 

We are left with the case $r=2$. But then $F/\Cc(T)$ is a cyclic extension with two branch points and hence $g=0$. This ends the proof 

\noindent
of the inequality $(G,r,{\bf C}) \prec (G_{T_0},r_{T_0},{\bf C}_{T_0})$ when $g\geq 1$.
%

Assume next $G_{T_0}=G$. The above inequality then also holds if $g=0$. The only non-trivial point is $r_{T_0} \geq r$.  We know that $r_{T_0} < r$ possibly happens only when $r\leq 3$ and $g=0$ and in this case $r_{T_0} < r$ means that either $r_{T_0}=0$ in which case $F_{T_0}=\Cc(T)$ and then $G_{T_0}=\{1\} \not=G$, or, $r_{T_0}=2$ in which case $F_{T_0}/\Cc(T)$ is cyclic, and then again $G_{T_0} \not=G$. Indeed, in this last case, $G$ cannot be cyclic as $r=3$, $g=0$ and a cyclic group has no generating set $\{g_1,g_2,g_3\}$ such that $g_1g_2g_3=1$ and with respective orders those in one of the above triples ${\bf e}$. 
Finally it immediately follows from theorem \ref{thm:r-in-spec} (d) that $g_{T_0}\geq g$ if $r\geq 4$.

\begin{remark}
If $F$ is of genus $0$ and $G_{T_0} \not=G$, $r_{T_0} < r$ may happen. One may then have $G_{T_0} =\{1\}$ or not (see example 3.3.2).


\end{remark}

\subsection{The exceptional genus $0$ cases} \label{ssec:exceptional-list}
The Riemann-Hurwitz formula
\vskip 0,5mm

\centerline{$\displaystyle 2g-2 = -2d + \sum_{i=1}^r (e_i-1) \frac{d}{e_i}$\hskip 5mm  where $d=|G|$,}
\vskip 0,5mm

\noindent
in a Galois situation yields
\vskip 1mm

\centerline{$2g-2 = d\hskip 1pt (r-2-\varepsilon)$}
\vskip 2mm

\noindent
As $\varepsilon \leq r/2$ we have $\displaystyle 2g-2 \geq d\hskip 1pt (\hskip 1pt \frac{r}{2}-2)$, and if $\varepsilon \leq (r-1)/2$, then

\noindent
we have $\displaystyle 2g-2 \geq \frac{d}{2} \hskip 1mm (r-3)$. Hence if $g=0$,  
$r\leq 3$ and $\varepsilon > 1$. 
\vskip 1mm

\noindent
Conclude that if $g=0$ we necessarily are in one of these cases:
\vskip 2mm

\noindent
{\rm (1)} {\it $r=3$ and ${\bf e} \in \{(2,2,2), (2,3,3), (2,3,4), (2,3,5), (2,2,n)$, $(n\geq 3) \}$}

\noindent
\hskip 2mm 
{\it with corresponding groups $(\Zz/2\Zz)^2$, $A_4$, $S_4$, $A_5$, $D_{2n}$ $(n\geq 3)$.}
\vskip 1mm

\noindent
 {\rm (2)} {\it $r=2$ then $F/\Cc(T)$ is a cyclic extension with $2$ branch points.}
\vskip 1,5mm

\noindent
Namely, a simple calculation shows that the tuples ${\bf e}$ are the indicated ones. Concerning the corresponding groups, note first that, as $F$ is of genus $0$, $G$ must be a subgroup of ${\rm PGL}_2(\Cc)$ and so one of the proposed list. The case $r=2$ is clear. Assume $r=3$. Then $G$ cannot be cyclic. For ${\bf e}=(2,2,2)$, $G$ is generated by two involutions with product an involution, so $G=(\Zz/2\Zz)^2$. Similarly one obtains $D_{2n}$ (dihedral group of order $2n$)  if ${\bf e}=(2,2,n)$ ($n\geq 3$). 
If ${\bf e}=(2,3,4)$ $G$ must be $S_4$ as it cannot be any of the others. We obtain similarly that $G=A_4$ if ${\bf e}=(2,3,3)$
and $G= A_5$ if ${\bf e}=(2,3,5)$.



Next we show that in all these cases, if $r$ distinct points  $t_1,\ldots,t_r \in \Pp^1(\Cc)$ are fixed
($r=2$ or $r=3$), 
%
%
%
there is one and only one Galois extension $F/\Cc(T)$ of group $G$,
ramification indices ${\bf e} = (e_1,\ldots,e_r)$ and branch points $t_1,\ldots,t_r$. Furthermore, as ${\rm PGL}_2$ is $3$-transitive on $\Pp^1(\Cc)$, up to isomorphism, there is exactly one extension $F/\Cc(T)$ in each case.

From corollary \ref{cor:genus0-param}, this unique extension $F/\Cc(T)$ of group $G$ in each case is $\Cc(U)$-parametric.

Concerning uniqueness, one checks first that up to some (anti-)isomor\-phism
of $G$ (which does not change the Galois extension $F/\Cc(T)$), there is, for each $r$-tuple ${\bf e}$, a unique possible $r$-tuple ${\bf C}=(C_1,\ldots,C_r)$ and second, that this  $r$-tuple is {\it rigid}, that is: there is a unique  $r$-tuple $(g_1,\ldots ,g_r)\in C_1\times \cdots \times C_r$ such that $\langle g_1,\ldots g_r\rangle = G$ and $g_1\cdots g_r=1$, up to componentwise conjugation by an element of $G$. It then classically follows from the Riemann Existence Theorem that there is one and only one Galois extension $F/\Cc(T)$ as desired if in addition the branch points are fixed.

Below we produce in each case an example of an extension $F/\Cc(T)$ with the given invariants.
\vskip -2mm

\noindent 
\subsubsection{} $r=2$, $G=\Zz/d\Zz$ with $d\geq 1$: $\Cc(\root{d}\of{T})/\Cc(T)$ is a Galois extension of group $\Zz/d\Zz$ branched at $0$ and $\infty$ with ramification indices $d$. 
\vskip -2mm

\noindent
\subsubsection{} ${\bf e}= (2,2,2)$, $G=(\Zz/2\Zz)^2$: for $F=\Cc(\sqrt{T}, \sqrt{T-1})$, $F/\Cc(T)$ is a Galois extension of group $(\Zz/2\Zz)^2$. A primitive element of $F/\Cc(T)$ is for example $\sqrt{T}+ \sqrt{T-1}$. An affine equation is the polynomial 
$Y^4 + 2(1-2T)Y^2+ 1$. There are three branch points: $0$, $1$ and $\infty$, which all are of index $2$. For $T_0= T^2$, we have $F_{T_0} = \Cc(T,\sqrt{T^2-1})$ whose branch points are $1$ and $-1$.

\vskip 1,5mm

\noindent
For the other cases, we produce a generating set of $G$ of $3$ elements $g_1,g_2,g_3$ of orders $e_1$, $e_2$, $e_3$ and such that $g_1 g_2 g_3 = 1$.

\subsubsection{} ${\bf e}= (2,3,3)$, $G=A_4$: take  
\vskip 0,5mm

\centerline{$g_1 = (1\hskip 2pt 2) (3 \hskip 2pt 4)$, $g_2 = (1\hskip 2pt 2 \hskip 2pt 3)$, $g_3 = (2\hskip 2pt 3 \hskip 2pt 4)$}

\subsubsection{}  ${\bf e}= (2,3,4)$, $G=S_4$: take 
\vskip 0,5mm

\centerline{
 $g_1 = (1\hskip 2pt 2)$, $g_2 = (2\hskip 2pt 3 \hskip 2pt 4)$, $g_3 = (4\hskip 2pt 3 \hskip 2pt 2 \hskip 2pt 1)$}
\vskip 1mm

\noindent
(The congugacy classes of $g_1,g_2,g_3$ in $S_4$ being ``rational'', a standard argument shows further that, if one fixes the three branch points in $\Pp^1(\Qq)$, the unique corresponding extension $F/\Cc(T)$ is defined over $\Qq$).

\subsubsection{}   ${\bf e}= (2,3,5)$,  $G=A_5$: take 
\vskip 0,5mm

\centerline{ $g_1 = (1\hskip 2pt 5) (3 \hskip 2pt 4)$, $g_2 = (1\hskip 2pt 2 \hskip 2pt 4)$, $g_3 = (5\hskip 2pt 4 \hskip 2pt 3 \hskip 2pt 2 \hskip 2pt 1)$}

\subsubsection{}   ${\bf e}= (2,2,n)$, $e\geq 3$, $G=D_{2n}$: take $g_1$, $g_2$ two involutions with  $g_1g_2=g_3^{-1}$ generating the normal cyclic subgroup. For $n$ odd, an explicit example is the Galois extension $F/\Cc(T)$ with affine equation $Y^{2n}-TY^n+1$ which is branched at $2$, $-2$, $\infty$ with ramification indices $2$, $2$ and $n$. As it is $\Cc(U)$-parametric, it follows from the uniqueness conclusion of 
theorem \ref{thm:order-invariants} (b) that it is isomorphic to the Hashimoto-Mihake generic extension for $D_{2n}$ mentioned in remark \ref{rem:param-gen}.

\vskip 1mm

\subsection{Theorem \ref{thm:order-invariants} (b)} \label{ssec:proof-continued}

\subsubsection{A preliminary lemma} \label{ssec:strict-growth} Retain the notation already introduced for theorem \ref{thm:order-invariants} (a).

\begin{lemma} \label{lem:r-strictly-grows-in-spec}  When $N=\deg(T_0)>1$, we have $r_{T_0} > r$ in each of the following cases:
\vskip 0,5mm

\noindent
{\rm (a)} $r\geq 5$,
\vskip 1mm

\noindent
{\rm (b)} $F\not=\Cc(T)$ and $\varepsilon \leq (r-2)/2$,
\vskip 1mm

\noindent
{\rm (c)} $N\geq 4$, $r=4$ and $\varepsilon \leq 3/2$,
\vskip 1mm

\noindent
{\rm (d)} $N\geq 4$, $r=3$ and $\varepsilon \leq 3/4$.
\end{lemma}

\begin{proof}
Assume $N>1$. If $r\geq 5$ as in (a), theorem \ref{thm:r-in-spec} (b-2) gives

\vskip 1,5mm

\centerline{$\displaystyle  r_{T_0}  \geq (r-4)N+4 \geq 2r-4 > r $}

\vskip 1,5mm

\noindent
From theorem \ref{thm:r-in-spec} (b-1), we have

\vskip 2mm

\noindent
(*) \hskip 30mm $\displaystyle  r_{T_0}  > (r-\varepsilon - 2) N +2$
\vskip 2mm

\noindent
Under the assumptions of (b), we deduce
\vskip 1,5mm

\centerline{$\displaystyle  r_{T_0}  > (\hskip 1mm \frac{r}{2}-1\hskip 1mm) N +2 \geq (\hskip 1mm \frac{r}{2}-1\hskip 1mm) 2 +2 = r$}
\vskip 1,5mm

\noindent
Finally $r_{T_0}>r$ easily follows from (*) above in the  last two cases (c) and (d). \end{proof}


%
%
%
%
%

%

\subsubsection{Proof of theorem \ref{thm:order-invariants} (b)} 
The only non-trivial point is the antisymmetry of $\specz$.
Let $F/\Cc(T)$ and $F^\prime/\Cc(T)$ be two non-isomorphic extensions in the set ${\mathcal E}^\ast$ such that $F/\Cc(T)\specz F^\prime/\Cc(T)$ and $F^\prime /\Cc(T)\specz F/\Cc(T)$. Let $T_0, T_0^\prime \in \Cc(T)$ such that $F^\prime/\Cc(T) = F_{T_0}/\Cc(T)$ and $F/\Cc(T) = F^\prime_{T^\prime_0}/\Cc(T)$ with $\deg(T_0) \geq 2$ and $\deg(T^\prime_0) \geq 2$. From theorem \ref{thm:order-invariants} (a), $F/\Cc(T)$ and $F^\prime/\Cc(T)$ have the same group $G$, the same branch point number $r$, the same inertia canonical invariant ${\bf C}$ and the same 
ramification indices ${\bf e}$.
We also have $F_{T_0T^\prime_0}/\Cc(T) = F/\Cc(T)$.
\vskip 0,5mm

Recall that $F/\Cc(T)\in {\mathcal E}^\ast$ means 
one of the following situations holds:
\vskip 0,5mm

\noindent
{\rm (a)} $G$ is of rank $\geq 4$,
\vskip 1mm

\noindent
{\rm (b)} $G$ is of rank $3$ and of odd order,
\vskip 1mm

\noindent
{\rm (c)} $G$ is of rank $2$ and order non divisible by $2$ or $3$,
\vskip 1mm


\noindent
{\rm (d)} $F$ is of genus $g=0$.
\vskip 1,5mm

\noindent
Each of the first three conditions further implies that\vskip 2mm

\noindent
(*) \hskip 3mm $r \geq 5$\hskip 2mm or \hskip 1mm(\hskip 1mm $\displaystyle r=4$ and $\displaystyle\varepsilon \leq \frac{4}{3}\leq \frac{3}{2}$\hskip 1mm )  \hskip 1mm or \hskip 1mm(\hskip 1mm $\displaystyle r=3$ and $\displaystyle\varepsilon \leq \frac{3}{5}\leq \frac{3}{4}$\hskip 1mm )
\vskip 2mm

\noindent
In the three cases, lemma \ref{lem:r-strictly-grows-in-spec} (applied with $N=\deg(T_0T^\prime_0)\geq 4$) yields a contradiction to $r_{T_0T^\prime_0} = r$. 

Suppose as in (d) that $F$ is of genus $0$. 
Then $F/\Cc(T)$ is one of the exceptional extensions described in \S \ref{ssec:exceptional-list}. But then so is $F^\prime/\Cc(T)$ as it has the same invariants $G$, $r$, ${\bf C}$, and again from \S \ref{ssec:exceptional-list}, it must be isomorphic to $F/\Cc(T)$, a contradiction. 
\vskip 0,5mm

For the second part of theorem \ref{thm:order-invariants} (b), fix a group $G\in {\mathcal G}^\ast$. If $G$ is not a subgroup of ${\rm PGL}_2(\Cc)$, all Galois extensions $L/\Cc(T)$ of group $G$ are in 
${\mathcal E}^\ast$ and a $\Cc(U)$-parametric extension $F/\Cc(T)$ of group $G$ is the smallest such extension for
the order $\prec$, hence is unique. If $G$ is a subgroup of ${\rm PGL}_2(\Cc)$, then $G$ has a $\Cc(U)$-parametric extension $F_{\rm m}/\Cc(T)$ (corollary \ref{cor:genus0-param}), which is an exceptional genus $0$ extension from \S \ref{ssec:exceptional-list}. If $F/\Cc(T)$ is another $\Cc(U)$-parametric extension of group $G$, it follows from $F_{\rm m}/\Cc(T) \prec F/\Cc(T)$ and $F/\Cc(T) \prec F_{\rm m}/\Cc(T)$ that  $F/\Cc(T)$ has the same invariants as $F_{\rm m}/\Cc(T)$, and so, as above, the two must be isomorphic. 

\subsubsection{An example}
Here is an example for which we have $(G,r,g,{\bf C}) =  (G_{T_0},r_{T_0},g_{T_0},{\bf C}_{T_0})$ but  $F/\Cc(T)$ and $F_{T_0}/\Cc(T)$ are not isomorphic, and so $N>1$. We do not know whether an example exists for which 
$F/\Cc(T)$ could additionally be shown to be a specialization of $F_{T_0}/\Cc(T)$ (which would show that the pre-order $\prec$ is not an order). 

Take $G=D_{2n}$ with $n$ odd and $F/\Cc(T)$ a Galois extension of group $G$ with branch points $0,1,-1,\lambda$ with $\lambda \in \Cc\setminus \{0,\pm1\}$ and ramification indices ${\bf e}=(2,2,2,2)$; such an extension exists from the RET and the easy construction of a generating set of $G$ of four elements $g_1,\ldots, g_4$ of order $2$ and such that $g_1\cdots g_4 = 1$.

Take $T_0(U) = U^2/(2U^2-2U+1)$. One checks that $T_0(u)=0$ has a double root, $u=0$, and that 
$T_0(u)=1$ has a double root, $u=1$. It follows that both $T_0(u)=-1$ and $T_0(u)=\lambda$ have two distinct roots (because of inequality (2) of the proof of theorem \ref{thm:r-in-spec}). From the analysis of the ramification in specialized extensions in the proof of theorem \ref{thm:r-in-spec} (more particularly from (*) and (**)), we obtain that 
$F_{T_0}/\Cc(T)$ has $r_{T_0}=4$ branch points, with ramification indices $2$.

The extensions $F/\Cc(T)$ and $F_{T_0}/\Cc(T)$ are not isomorphic. Otherwise the cross-ratio of their branch points would be equal, up to the sign. The cross-ratio of $0,1,-1,\lambda$ is $(\lambda-1)/(2\lambda)$. The branch points of $F_{T_0}/\Cc(T)$ are the simple roots of $T_0(u)=1$ and $T_0(u)=\lambda$. Take for example $\lambda = 1/5$. These four points are then $(1\pm \sqrt{-2})/3$, $-1$ and $1/3$. A final computation shows the corresponding cross-ratio is $(16+4\sqrt{-2})/9$ while $(\lambda-1)/(2\lambda) = -2$.

Assume $G_{T_0}\not=G$. From theorem \ref{thm:r-in-spec} (a), there exists a sub-extension $L/\Cc(T)$ of $F/\Cc(T)$ such that $L\not=\Cc(T)$, $L_{T_0} = \Cc(T)$ and $L$ of genus $0$. Write $L=\Cc(\theta)$ for some $\theta\in F$ and $T=A(\theta)/B(\theta)$ with $A,B \in \Cc[X]$ re\-la\-ti\-vely prime, $B\not=0$. The irreducible polynomial of $\theta$ over $\Cc(T)$ is $A(Y)-TB(Y)$. Then $L_{T_0} = \Cc(U)$ means that $A(Y)-T_0(U)B(Y)$ has a root $Y_0(U)\in \Cc(U)$. But then we have $T_0(U) = A(Y_0(U))/B(Y_0(U))$. As we explain in the last paragraph, $T_0$ is indecomposable so necessarily $A/B = T_0$ and so $L$ does not depend on $\lambda$. The next paragraph provides a contradiction by showing that $L$ is ramified over $\lambda$.

The Galois group $\Gal(F/L)$ cannot be a subgroup of the cyclic subgroup of order $n$ of $D_{2n}$: otherwise, with $d_L=[L:\Cc(T)]$, the Riemann-Hurwitz formula yields $-2 = -2 d_L + 4d_L/2$, a contradiction. Therefore $L$ is the fixed field in $F$ of some involution of $D_{2n}$. The Riemann-Hurwitz formula gives: $-2=-2n+ R$ where $R$ is the number of ramified primes. Conclude that above each of $0,1,-1, \lambda$, the number of ramified points is the maximum possible: $(n-1)/2$.

That $T_0$ is indecomposable is an exercise for which we only indicate the main steps.
Deduce from $T_0(U)=A(Y_0(U))/B(Y_0(U))$ that $A(Y_0(U))=K(U) U^2$ and $B(Y_0(U))=K(U)(2U^2-2U+1)$ for some $K\in \Cc(U)$. Writing $Y_0(U)=\alpha(U)/\beta(U)$ with $\alpha,\beta \in \Cc[U]$ relatively prime, show next that necessarily $Y_0(U)\in \Cc[U]$ and $K(U)\in \Cc$. The desired conclusion easily follows.

\section{Twisting regular Galois extensions in families} \label{sec:twisting} 
Here $k$-regular extensions $F/k(T)$ are viewed as fundamental group representations, as recalled in \S \ref{ssec:fund-group-basics}. \S \ref{ssec:twisting} recalls the twisting operation on covers and the twisting lemma (\S \ref{ssec:twisting_lemma}). 
\S \ref{ssec:twisting-families} explains how the twisting lemma can be used ``in family''. \S \ref{ssec:statement-main-thm} states theorem \ref{thm:main}, which the main result of this section; theorem \ref{thm:twist-main} is a special case. Theorem \ref{thm:main} is proved in \S \ref{ssec:proof-main-thm}. 



\subsection{Fundamental groups representations} \label{ssec:fund-group-basics}
%
%
The absolute Galois group of a field $K$ is denoted by $\Gabs_K$. 
If $E/K$ is a Galois extension of group $G$, an epimorphism $\varphi: \Gabs_K\rightarrow G$ such that $E$ is the fixed field of ${\rm ker}(\varphi)$ in $\overline K$ is a called a {\it $\Gabs_K$-representation} of $E/K$.

Given a finite subset ${\mathbf t}\subset \Pp^1(\overline K)$ invariant under $\Gabs_K$, the $K$-funda\-men\-tal group of $\Pp^1\setminus {\mathbf t}$ is denoted by $\pi_1(\Pp^1\setminus {\mathbf t}, t)_K$; here $t$ denotes a fixed {\it base point}, which corresponds to choosing an embedding of $K(T)$ in an algebraically closed field $\Omega$. The (geometric) $\overline K$-fundamental group $\pi_1(\Pp^1\setminus {\mathbf t}, t)_{\overline K}$ is defined as the Galois group of the maximal algebraic extension $\Omega_{{\mathbf t},K}/\overline K(T)$ (inside $\Omega$) unramified above $\Pp^1\setminus {\mathbf t}$ and the (arithmetic) $K$-fundamental group $\pi_1(\Pp^1\setminus {\mathbf t}, t)_{K}$ as the group of the Galois extension $\Omega_{{\mathbf t},k}/K(T)$. 

Degree $d$ $K$-regular extensions $F/K(T)$ (resp. $K$-regular Galois extensions $F/K(T)$ of group $G$) with branch points in $\mathbf t$ correspond to transitive homomorphisms $\pi_1(\Pp^1\setminus {\mathbf t}, t)_{K} \rightarrow S_d$ (resp. to epimorphisms $\pi_1(\Pp^1\setminus {\mathbf t})_{K},t) \rightarrow G$), with the extra regularity condition that the restriction of $\phi$ to $\pi_1(\Pp^1\setminus {\mathbf t})_{K},t)_{\overline K}$ remains transitive (resp. remains onto). These corresponding homomorphisms are called the {\it fundamental group representations} (or {\it $\pi_1$-representations} for short) of the $K$-regular (resp. $K$-regular Galois) extension $F/K(T)$.

Each $K$-rational point $t_0\in \Pp^1(K)\setminus {\mathbf t}$ naturally provides a section
${\sf s}_{t_0}: \Gabs_K\rightarrow \pi_1(\Pp^1\setminus {\mathbf t},t)_{K}$ to the exact sequence
\vskip 1,5mm

\centerline{$ 1\rightarrow \pi_1(\Pp^1\setminus {\mathbf t}, t)_{\overline K} \rightarrow \pi_1(\Pp^1\setminus {\mathbf t},t)_K \rightarrow \Gabs_K \rightarrow 1$} 
\vskip 1,5mm

\noindent
which is uniquely defined up to conjugation by an element in the fundamental group $\pi_1(\Pp^1\setminus {\mathbf t}, t)_{\overline K}$.

 If $\phi: \pi_1(\Pp^1\setminus {\mathbf t}, t)_K \rightarrow G$ represents a $K$-regular Galois extension $F/K(T)$, the morphism $\phi \circ {\sf s}_{t_0}:\Gabs_K \rightarrow G$ is the {\it specialized representation} of $\phi$ at $t_0$. The fixed field in $\overline K$ of ${\rm ker}(\phi \circ {\sf s}_{t_0})$ is the specialized extension $F_{t_0}/K$ of $F/K(T)$
 at $t_0$.
\vskip 3mm

\subsection{Twisting regular Galois extensions} \label{ssec:twisting}

For this subsection, we refer to \cite{DEGha}.

\subsubsection{The twisting lemma} \label{ssec:twisting_lemma}
Fix a field $K$ and a $K$-regular Galois extension ${\mathcal F}/K(T)$ 
of group $G$, also viewed as $K$-regular Galois cover $f:X\rightarrow \Pp^1$. Recall how it can be twisted by a 
Galois extension $E/K$ of group $H\subset G$. Formally this is done in terms of the associated $\pi_1$- 
and $\Gabs_K$- representations.  

Let $\phi: \pi_1(\Pp^1\setminus {\mathbf t}, t)_K \rightarrow G$ be a $\pi_1$-representation of ${\mathcal F}/K(T)$ and $\varphi:\Gabs_K \rightarrow G$ be a $\Gabs_K$-representation of  the Galois extension $E/K$.

Denote the right-regular (resp. left-regular) representation of $G$ by $\delta: G\rightarrow S_d$ (resp. by $\gamma: G\rightarrow S_d$) where $d=|G|$. 
Consider the map 
\vskip 1mm

\centerline{$\widetilde \phi^\varphi:
\pi_1(\Pp^1\setminus {\mathbf t}, t)_K \rightarrow S_d$} 
\vskip 1mm

\noindent
defined by the following formula, where $R$ is the restriction map $\pi_1(\Pp^1\setminus {\mathbf t}, t)_K \rightarrow \Gabs_K$ and $\times$ is the multiplication in the symmetric group $S_d$:
\vskip 2mm

\noindent
(*)  \hskip 12mm $\widetilde \phi^\varphi (\Theta)  = \gamma(\phi (\Theta)) \times \delta (\varphi (R(\Theta))^{-1}) \hskip 6mm (\Theta\in \pi_1(\Pp^1\setminus {\mathbf t}, t)_K). $
\vskip 2mm

\noindent
The map $\widetilde \phi^\varphi$ is a group homomorphism
with the same restriction on $\pi_1(\Pp^1\setminus {\mathbf t}, t)_{\overline K}$
as $\phi$. It is called  the {\it twisted representation} of 
$\phi$ by $\varphi$. 

The associated $K$-regular extension is denoted by $\widetilde {\mathcal F}^\varphi/K(T)$ and  called  the {\it twisted extension} of 
${\mathcal F}/K(T)$ by $\varphi$. The field $\widetilde {\mathcal F}^\varphi$ is the fixed field in $\Omega_{{\mathbf t},K}$ of the subgroup $\Gamma\subset \pi_1(\Pp^1\setminus {\mathbf t}, t)_{{K}}$ of all $\Theta$ such that $\widetilde \phi^\varphi (\Theta)$ fixes the neutral element of $G$ \footnote{Taking any other element of $G$ gives the same field up to $K(T)$-isomorphism.}. Finally the corresponding $K$-regular cover, the {\it twisted cover} of 
$f$ by $\varphi$, is denoted by $\widetilde f^\varphi: \widetilde X^\varphi \rightarrow \Pp^1$.


The following statement is the main property 
of the twisted cover.

\begin{twisting lemma} \label{prop:twisted cover} 
Let $t_0\in \Pp^1(K)\setminus {\mathbf t}$. The specialization representation $\phi \circ {\sf s}_{t_0}:\Gabs_K \rightarrow G$ 
is conjugate in $G$ to $\varphi:\Gabs_K \rightarrow G$ if and only if there exists $x_0\in \widetilde X^\varphi(K)$ such that $\widetilde f^\varphi(x_0)=t_0$.
\end{twisting lemma}

As a first illustration, we prove the following statement, which we have alluded to several times.

\begin{corollary} \label{cor:genus0-param}
If  $F/\Cc(T)$ is a Galois extension of group $G$ with $F$ of genus $0$, then
$F/\Cc(T)$ is $\Cc(U)$-parametric.
\end{corollary}

\begin{proof}
Let $F/\Cc(T)$ as above and $L/\Cc(U)$ be a Galois extension of group $H\subset G$. Let $\phi: \pi_1(\Pp^1\setminus {\mathbf t}, t)_{\Cc(U)} \rightarrow G$ be a $\pi_1$-representation of $F\Cc(U)/\Cc(U)(T)$ and $\varphi:\Gabs_{\Cc(U)} \rightarrow H\subset G$ be a $\Gabs_{\Cc(U)}$-representation of  the Galois extension $L/\Cc(U)$.
Set ${\mathcal F}=F\Cc(U)$ and consider the twisted extension $\widetilde {\mathcal F}^\varphi/\Cc(U)(T)$ and the associated twisted cover $\widetilde X^\varphi \rightarrow \Pp^1$. The extension ${\mathcal F}\overline{\Cc(U)}/\overline{\Cc(U)}(T)$ and $F\overline{\Cc(U)}/\overline{\Cc(U)}(T)$ are $\overline{\Cc(U)}(T)$-isomorphic. Consequently $\widetilde X^\varphi$ has the same genus as $F$, that is $0$. From Tsen's theorem, $\widetilde X^\varphi$ has a $\Cc(U)$-rational point and is isomorphic to $\Pp^1$ over $\Cc(U)$. Conclude thanks to lemma \ref{prop:twisted cover}  that $L/\Cc(U)$ is 
a $\Cc(U)$-specialization of $F/\Cc(T)$.
\end{proof}

It is a similar argument that proves that if $K$ is a Pseudo Algebraically Closed field, then every $K$-regular Galois extension $F/K(T)$ is $K$-parametric \cite[\S 3.3.1]{DeBB}.

\subsection{Twisting in families} \label{ssec:twisting-families}

Consider the twisted extension  $\widetilde {\mathcal F}^\varphi/K(T)$ when $K=k(U)$ with $k$ a field and $U$ some indeterminate.

\subsubsection{Description of the twisted extension}  \label{ssec:description-twisted}
%
%
Every element $\Theta$ in the $K$-fundamental group $\pi_1(\Pp^1\setminus {\mathbf t}, t)_K$ uniquely writes $ \Theta = \chi \hskip 2pt {\sf s}_U(\sigma)$ with $\chi \in \pi_1(\Pp^1\setminus {\mathbf t}, t)_{\overline{K}}$ and $\sigma\in \Gabs_{K}$.
Whence
\vskip 1,5mm

\centerline{$\left\{
\begin{matrix} 
\phi (\Theta)  = \phi (\chi) \hskip 2pt \phi ({\sf s}_U(\sigma)) \hfill \cr
\varphi (R(\Theta)) = \varphi(\sigma) \hfill \cr
\end{matrix}
\right.
$}
\vskip 1,5mm

\noindent
and the following formula, where, by ${\rm conj}(g)$  ($g\in G$),  we denote the permutation of $G$ induced by the conjugation $x\rightarrow gxg^{-1}$:
\vskip 2mm

\centerline{$\widetilde \phi^\varphi (\Theta)  = \gamma(\phi (\chi) \hskip 2pt \phi ({\sf s}_U(\sigma)) \hskip 2pt \varphi(\sigma)^{-1}) \times {\rm conj}(\varphi(\sigma)).$}
\vskip 2mm

\noindent
Conclude that the field $\widetilde {\mathcal F}^\varphi$ is the fixed field in $\Omega_{{\mathbf t},K}$ of the following subgroup $\Gamma\subset \pi_1(\Pp^1\setminus {\mathbf t}, t)_{{K}}$:

\vskip 2mm

\centerline{$\Gamma=\{\chi \hskip 1pt {\sf s}_U(\sigma) \in  \pi_1(\Pp^1\setminus {\mathbf t}, t)_{{K}}
\hskip 2pt | \hskip 2pt \phi (\chi) = \hskip 2pt \varphi(\sigma)\hskip 1pt \phi ({\sf s}_U(\sigma)) ^{-1}
\}$}
\vskip 2mm

%
%
%
%
%
%
%

\subsubsection{The fiber-twisted cover at $u_0$} \label{ssec:fiber-twisted} 
The two extensions ${\mathcal F}/K(T)$ and $\widetilde{{\mathcal F}}^{\varphi}/K(T)$ are $K$-regular. From the Grothendieck good reduction theorem, for every $u_0\in k$ but in a finite subset ${\mathcal E}$, 
they specialize at $U=u_0$ to respective extensions ${\mathcal F}|_{u_0}/k(T)$ and $(\widetilde {\mathcal F}^\varphi)|_{u_0}\hskip 2pt /\hskip 1pt k(T)$ that are $k$-regular of degree
\vskip 1,5mm
 
\centerline{$[\widetilde{{\mathcal F}}^{\varphi}:k(U)(T)]=[{\mathcal F}:K(T)] = d$,}
\vskip 1,5mm

\noindent
have branch point set ${\bf t}_{u_0}$ and have the same genus as the common genus of the function fields ${\mathcal F}$ and $\widetilde {\mathcal F}^\varphi$. 

Using the embedding of $\overline{k(U)}$ in the field of Puiseux series in $U-u_0$ and coefficients in $\overline k$, we have a natural monomorphism
\vskip 2mm
 
\centerline{${\sf s}_{u_0}: \Gabs_k \rightarrow \Gabs_{K}$}
\vskip 2mm

\noindent
The morphism $\varphi \circ {\sf s}_{u_0}:\Gabs_k \rightarrow G$ is well-defined and so is the twisted extension 
\vskip 1mm
 
\centerline{$\widetilde{\left({\mathcal F}|_{u_0}\right)}^{\varphi\circ  {\rm s}_{u_0}}/\hskip 1pt k(T)$}
\vskip 1,5mm

\noindent
We call it the {\it fiber-twisted extension} at $u_0$. If $\phi |_{u_0} : \pi_1(\Pp^1\setminus {\bf t}_{u_0},t)_k \rightarrow G$ is a $\pi_1$-representation of the $k$-regular Galois extension ${\mathcal F}|_{u_0}/k(T)$, then a $\pi_1$-representation of the  twisted extension above is 
\vskip 2mm 
 
\centerline{$\widetilde{\phi |_{u_0}}^{\varphi \circ {\sf s}_{u_0}}: \pi_1(\Pp^1\setminus {\mathbf t}_{u_0}, t)_{k} \rightarrow S_d$}
\vskip 2mm

%
%
%

Every element $\theta\in \pi_1(\Pp^1\setminus {\mathbf t}, t)_{k}$ uniquely writes $\theta = x\hskip 2pt {\sf s}_{u_0}(\tau)$ with $x \in \pi_1(\Pp^1\setminus {\mathbf t}, t)_{\overline{k}}$ and $\tau \in \Gabs_{k}$. Similarly as in \S \ref{ssec:description-twisted} 
we obtain:

\vskip 3mm

\centerline{$\widetilde{\phi |_{u_0}}^{\varphi \circ {\sf s}_{u_0}} (\theta)  = \gamma\left(\phi |_{u_0}(x) \hskip 1pt \phi |_{u_0}({\sf s}_{u_0}(\tau)\right) \hskip 1pt \varphi({\sf s}_{u_0}(\tau))^{-1}) \times {\rm conj}(\varphi({\sf s}_{u_0}(\tau)))$}
\vskip 3mm

\noindent
The field $\widetilde{\left({\mathcal F}|_{u_0}\right)}^{\varphi\circ  {\rm s}_{u_0}}$ is the fixed field in $\Omega_{{\mathbf t}_{u_0},k}$ of the following subgroup $\Gamma_{u_0}$ of $\pi_1(\Pp^1\setminus {\mathbf t}_{u_0}, t)_{{k}}$:


\vskip 3mm

\centerline{$\Gamma_{u_0}=\{x\hskip 1pt {\sf s}_{u_0}(\tau) 
\hskip 3pt | \hskip 3pt \phi |_{u_0}(x) = \hskip 2pt \varphi ({\sf s}_{u_0}(\tau)) \hskip 1pt \phi |_{u_0}({\sf s}_{u_0}(\tau))^{-1}
\}$}
\vskip 3mm


%
%
\subsection{Comparison statement} \label{ssec:statement-main-thm} 
\begin{theorem} \label{thm:main}
For all but finitely many $u_0\in k$, the two extensions 
\vskip 2mm 

\centerline{$(\widetilde {\mathcal F}^\varphi)|_{u_0}\hskip 2pt /\hskip 1pt k(T)$ and  $\widetilde{\left({\mathcal F}|_{u_0}\right)}^{\varphi\circ  {\rm s}_{u_0}}/\hskip 1pt k(T)$}
\vskip 2mm

\noindent
are well-defined and are $k(T)$-isomorphic.
\end{theorem}


\begin{corollary}  \label{cor:twist-main}
Let $\widetilde P^\varphi \in k[U,T,Y]$ be an affine equation of $\widetilde {\mathcal F}^\varphi/K(T)$. 
For all but finitely many $u_0\in k$, the polynomial 
\vskip 2mm 

\centerline{$\widetilde P^\varphi(u_0,T,Y)$}
\vskip 1mm
 
\noindent
is an affine equation of the $k$-regular extension $\widetilde{\left({\mathcal F}|_{u_0}\right)}^{\varphi\circ  {\rm s}_{u_0}}/\hskip 1pt k(T)$. Consequently, for all but finitely many $u_0\in k$ and for all $t_0 \notin ({\bf t}|_{u_0}\cup \{\infty\})$, we have this criterion
\vskip 1,5mm

\noindent
{\rm (*)} {\it there exists $y_0\in k$ such that $\widetilde P^\varphi(u_0,t_0,y_0)=0$ if and only if the specialization representation $\phi |_{u_0} \circ {\rm s}_{t_0}$  is conjugate in $G$ to $\varphi\circ  {\rm s}_{u_0}$.}
\end{corollary}

\begin{proof}[Proof of theorem \ref{thm:twist-main}] Theorem \ref{thm:twist-main} is a special case of corollary \ref{cor:twist-main}. Namely, from the two $k$-regular Galois extensions $F/k(T)$ and $L/k(T)$
given in theorem \ref{thm:twist-main}, consider the $K$-regular extension ${\mathcal F}/K(T)$ deduced from $F/k(T)$ by scalar extension from $k$ to $K=k(U)$, and let $\varphi:\Gabs_K \rightarrow G$ be a $\Gabs_K$-representation of the extension obtained by specializing $LK/K(T)$ at $T=U\in K$. Corollary \ref{cor:twist-main} applied for this ${\mathcal F}/K(T)$ and this $\varphi:\Gabs_K \rightarrow G$ yields theorem \ref{thm:twist-main}.

Furthermore, because ${\mathcal F}/K(T)$ is obtained by scalar extension from an extension of $k(T)$, the set ${\mathcal E}$ of bad $u_0$ in $k$ is a finite subset of $k$.

Finally the appearance of several polynomials $\widetilde P_1,\ldots,\widetilde P_n$ in theorem \ref{thm:twist-main} comes the fact that its statement is phrased in terms of field extensions rather than in group 
representations. In the generality of corollary \ref{cor:twist-main}, we have
\vskip 1,5mm

\noindent
(***) {\it the field extension $E|_{u_0}/k$ is the specialization of ${\mathcal F}|_{u_0}/k(T)$ at $t_0$ if and only if there exists $\chi \in {\rm Aut}(G)$ such that $\phi |_{u_0} \circ {\rm s}_{t_0}$  is conjugate in $G$ to $\chi \circ \varphi\circ  {\rm s}_{u_0}$.}
\vskip 1,5mm

\noindent
and so $\widetilde P_1,\ldots,\widetilde P_n$ are the polynomials $\widetilde P^{\chi\circ \varphi}$
with $\chi \in {\rm Aut}(G)$.
\end{proof}

%
%
%

\subsection{Proof of theorem \ref{thm:main}} \label{ssec:proof-main-thm}
%
%

%
%
Let ${\mathcal E}$ be the finite subset given by the Grothendieck good reduction theorem (\S \ref{sec:preliminaries}). Fix $u_0\in k\setminus {\mathcal E}$. The two extensions from the statement of theorem \ref{thm:main}
are well-defined and have the same branch point set, namely ${\bf t}_{u_0}$. We will show that they are $k(T)$-isomorphic by showing that they have the same $\pi_1$-representations. We need to compare $\widetilde{\phi |_{u_0}}^{\varphi \circ {\sf s}_{u_0}}$ from \S \ref{ssec:fiber-twisted} and some $\pi_1$-representation, say 
\vskip 1,5mm

\centerline{$\widetilde {\phi}^\varphi|_{u_0}: \pi_1(\Pp^1\setminus {\mathbf t}_{u_0}, t)_{k} \rightarrow S_d$}
\vskip 1,5mm

\noindent
of the $k$-regular extension $(\widetilde {\mathcal F}^\varphi)|_{u_0}\hskip 2pt /\hskip 1pt k(T)$.

As a first step, consider the restrictions of these $\pi_1$-representations to the geometric fundamental group $\pi_1(\Pp^1\setminus {\mathbf t}_{u_0}, t)_{\overline k}$. Recall that from the addendum to the Grothendieck good reduction theorem (\S \ref{sec:preliminaries}), we have a specialization isomorphism
\vskip 1,5mm

\centerline{${\rm sp}_{u_0}: \pi_1(\Pp^1\setminus {\mathbf t}, t)_{\overline{K}} \rightarrow \pi_1(\Pp^1\setminus
 {\mathbf t}_{u_0},t)_{\overline{k}}$}
\vskip 1,5mm

\noindent
and that for all $x\in \pi_1(\Pp^1\setminus {\mathbf t}, t)_{\overline{K}}$, we have
\vskip 2mm

\centerline{$\left\{
\begin{matrix}
\phi |_{u_0} (x) = \phi \circ {\rm sp}_{u_0}^{-1} (x) \hfill \cr
\widetilde {\phi}^\varphi |_{u_0} (x) = \widetilde {\phi}^\varphi \circ {\rm sp}_{u_0}^{-1} (x)
\end{matrix}
\right.$}

\vskip 2mm
\noindent
Using \S \ref{ssec:description-twisted}, we obtain
\vskip 2mm

\centerline{$\widetilde {\phi}^\varphi |_{u_0} (x) = \gamma (\phi({\rm sp}_{u_0}^{-1} (x)))= \gamma (\phi |_{u_0} (x))=\widetilde{\phi |_{u_0}}^{\varphi \circ {\sf s}_{u_0}} (x)$}

\vskip 2mm

To compare the restrictions to $\Gabs_k$ of the two $\pi_1$-representations, first show the following.

\begin{lemma} \label{lem:arith-action}
For all but finitely many $u_0\in k$ and all $\tau \in \Gabs_k$,  we have
\vskip 1,5mm

\centerline{$\phi |_{u_0} ({\sf s}_{u_0}(\tau)) = \phi({\sf s}_U \circ {\sf s}_{u_0}(\tau))$}

\end{lemma}

\begin{proof}
Namely, with ${\mathcal Y}_1$ a primitive element of ${\mathcal F}/K(T)$, which we may assume to be integral over $k[U,T]$, the right-hand side term corresponds to the action of ${\sf s}_{u_0}(\tau) \in \Gabs_K$ on the $d$ different $K$-conjugates 
\vskip 1mm
 
\centerline{$\displaystyle {\mathcal Y}_i = \sum_{n=0}^\infty b_{in}(U) (T-U)^n$, \hskip 3mm $j=1,\ldots,d$}
\vskip 1mm

\noindent
of ${\mathcal Y}_1$, viewed in $\overline K((T-U))$; the action of ${\sf s}_{u_0}(\tau)$ is given by the action on the coefficients $b_{in}(U)\in \overline{K}$ ($n\geq 0$). 

From the Eisenstein theorem,
there exists a polynomial $E(U) \in k[U]$, $E(U)\not=0$, such that $E(U)^{n+1} \hskip 2pt b_{in}(U)\in k[U]$ for every $n\geq 0$, $i=1,\ldots,d$. Enlarge again the set ${\mathcal E}$ to contain the roots of $E(U)$. Then ${\mathcal Y}_1,\ldots, {\mathcal Y}_d$ can be specialized at $U=u_0$ to yield $d$ formal power series in $\overline k[[T-u_0]]$
\vskip 1mm
 
\centerline{$\displaystyle {\mathcal Y}_i |_{u_0}= \sum_{n=0}^\infty b_{in}(u_0) (T-u_0)^n$, \hskip 3mm $j=1,\ldots,d$}
\vskip 1mm

If ${\mathcal E}$ is again enlarged to contain the roots of the bad prime divisor of the irreducible polynomial $P\in k[U,T,Y]$ of ${\mathcal Y}_1$ over $k(U,T)$ (\S \ref{sec:preliminaries}), then $P(u_0,T,Y)$ is irreducible in $\overline k[T,Y]$; it is the irreducible polynomial of ${\mathcal Y}_1|_{u_0}$ and the extension $k(T,{\mathcal Y}_1|_{u_0})/k(T)$ is $k(T)$-isomorphic to the
extension ${\mathcal F}|_{u_0}/K(T)$.

The left-hand side term $\phi |_{u_0} ({\sf s}_{u_0}(\tau))$ of the claimed formula corresponds to the action of $\tau \in \Gabs_k$ on the $d$ different $K$-conjugates  ${\mathcal Y}_1 |_{u_0},\ldots,  {\mathcal Y}_d|_{u_0}$, with $\tau$ acting on the coefficients $b_{in}(u_0)\in \overline{k}$ ($n\geq 0$). Clearly we have 
\vskip 1,5mm
 
\centerline{$\displaystyle \hskip 15mm \left({\sf s}_{u_0}(\tau)(b_{in}(U)\right)|_{u_0} = \tau(b_{in}(u_0))$, \hskip 3mm ($i=1,\ldots,d$, $n\geq 0$)}
\vskip 1mm

\noindent
and so
\vskip 1,5mm
 
\centerline{$\displaystyle \hskip 15mm ({\sf s}_{u_0}(\tau)({\mathcal Y}_i)|_{u_0} = \tau({\mathcal Y}_i |_{u_0})$, \hskip 3mm ($i=1,\ldots,d$)}
\vskip 1,5mm

\noindent
which corresponds to the claim. \end{proof}

Lemma \ref{lem:arith-action}, applied with $\widetilde {\phi}^\varphi$ replacing ${\phi}$ also gives
\vskip 1,5mm
 
\centerline{
$\widetilde {\phi}^\varphi |_{u_0} ({\sf s}_{u_0}(\tau)) = \widetilde {\phi}^\varphi({\sf s}_U \circ {\sf s}_{u_0}(\tau))$}
\vskip 1,5mm

\noindent
Using \S \ref{ssec:description-twisted}, we obtain
\vskip 2mm

\centerline{
$\begin{matrix}
\widetilde {\phi}^\varphi |_{u_0} ({\sf s}_{u_0}(\tau)) & = \gamma (\phi({\sf s}_U \circ {\sf s}_{u_0}(\tau))  \hskip 2pt \delta(\varphi({\sf s}_{u_0}(\tau)) \hfill \cr
& = \gamma (\phi |_{u_0} ({\sf s}_{u_0}(\tau)) \hskip 2pt \delta(\varphi({\sf s}_{u_0}(\tau)) \hfill \cr
& = \widetilde{\phi |_{u_0}}^{\varphi \circ {\sf s}_{u_0}} ({\sf s}_{u_0}(\tau)) \hfill \cr
\end{matrix}
$}
\vskip 3mm

\noindent
This concludes the proof of $\widetilde{\phi |_{u_0}}^{\varphi \circ {\sf s}_{u_0}} =\widetilde {\phi}^\varphi|_{u_0}$ and so of theorem \ref{thm:main}.

\bigskip

\section{Appendix: Good reduction \& specializations of covers} \label{sec:preliminaries}
This appendix recalls some classical results which essentially go back to Grothendieck about the good reduction of $K$-regular extensions and the inertia in their specializations. We have adjusted to our situation the original statements which hold in a bigger generality; in particular our statements are phrased in field extension terms rather than in a scheme theoretic language. 
Assume $K$ is the fraction field of a Dedekind domain $A$; typically $K=k(U)$ and $A=k[U]$ with $U$ a new indeterminate. 

Given a non-zero prime ideal ${\frak p}\subset A$ (typically ${\frak p}=\langle U-u_0 \rangle$ with $u_0\in k$ when $A=k[U]$), denote the residue field by $\kappa_{\frak p}$, 
the completion of $A$ (resp. of $K$) at ${\frak p}$ by $\widetilde A_{\frak p}$ (resp. by $\widetilde K_ {\frak p}$), 
the algebraic closure  of $\widetilde K_ {\frak p}$ by $C_{\frak p}$ and fix an embedding $\overline K\subset C_{\frak p}$.

 Let $F/K(T)$ be a $K$-regular extension of group $G$, with branch point set ${\bf t}=\{t_1,\ldots,t_r\}$, inertia canonical invariant ${\bf C}=(C_1,\ldots,C_r)$ and associated ramification indices ${\bf e} =(e_1,\ldots,e_r)$. Let ${\frak p}\subset A$ be a non-zero prime ideal. We recall below some classical results about 
\vskip 0,5mm

\noindent
(a) the good reduction of $F/K(T)$ modulo the prime ideal ${\frak p}$, and,
\vskip 0,5mm

\noindent
(b) the ramification above  the prime ideal ${\frak p}$ in specializations $F_{t_0}/K$ at points $t_0\in \Pp^1(K)$.
\vskip 0,5mm

These results go back to general results of Grothendieck \cite{SGA}, \cite{GrMu} and more specific versions by Beckmann for regular extensions $F/K(T)$  over number fields \cite{Beckmann-specialization}. Here, regarding (b), we follow 
Legrand's variant \cite[\S 2]{legrand1} extending Beckmann's statement to the situation the ground field is the fraction field of an arbitrary Dedekind domain. For (a) we follow the variant given in \cite{acta2016}.
\vskip 2mm 

\noindent
{\bf Classical assumptions.}
We first list some classical assumptions involved in these statements; we refer to the articles cited above for more 
details about them. The main point we will use is that each of them is satisfied for all but finitely many primes ${\frak p}$.


\vskip 0,5mm

\noindent
(1) $|{\mathcal G}|\notin {\frak p}$
\vskip 0,5mm
\noindent

\noindent
(2)  {\it there is no vertical ramification at ${\frak p}$ in the extension $F/K(T)$. }
\vskip 0,5mm

\noindent
(3) {\it no two different branch points of $F/K(T)$ meet modulo ${\frak p}$.}
\vskip 0,5mm

\noindent
(4) {\it the ideal ${\frak p}$ is unramified in the extension $K(t_1,\ldots,t_r)/K$.}

\noindent
(5) {\it $t_i$ and $1/t_i$ are integral over $\widetilde A_{\frak p}$, $i=1,\ldots, r$.}
\vskip 1,5mm

We will say that ${\frak p}$ is a {\it bad prime of the extension $F/K(T)$} if conditions (2), (3) hold\footnote{Legrand also includes (1) and (4). For consistency with the other good/bad prime notion, we prefer to stick to (2) and (3) and repeat (1) and (4) when necessary.} and that it is {\it good} otherwise. 

We also recall the related notion of {\it good/bad primes of a non-constant polynomia}l $P\in A[T,Y]$, irreducible in $\overline K[T,Y]$ and monic in $Y$, defined in \cite{acta2016}: a non-zero element ${\mathcal B}_P \in A$ is constructed and called the bad prime divisor of $P$; it is essentially the discriminant w.r.t $T$ of some ``reduced form'' of the discriminant $\Delta_P(T)$ of $P$ w.r.t. $Y$. A prime ${\frak p}$ is said to be a {\it good prime of $P$} if
\vskip 1,5mm

\noindent
(6) ${\mathcal B}_P \notin{\frak p}$.
\vskip 1mm

\noindent
Again there are only finitely many {\it bad primes} for the polynomial $P$. The two notions compare as follows: if ${\frak p}$ is good for $P$ then it is also good for the extension $K(T)[Y]/\langle P\rangle$ of $K(T)$.

 Let $B$ be the integral closure of $\widetilde A_{\frak p}[T]$ in the field $F\widetilde K_ {\frak p}$.

\vskip 2mm

\noindent
{\bf Grothendieck good reduction theorem.} {\it Assume that ${\frak p}$ is a good prime of $F/K(T)$ and that assumption {\rm (1)} holds.
Then the extension $F/K(T)$ has {\rm good reduction} at ${\frak p}$, i.e.: ${\frak p} B$ is a prime ideal of $B$ and the fraction field $\varepsilon$ of $B/{\frak p} B$ is a separable extension of $\kappa_{\frak p}(T)$ and satisfies 
\vskip 1mm

\centerline{$[\varepsilon :\kappa_{\frak p}(T)]=[\overline{\kappa_{\frak p}}\hskip 2pt \varepsilon :\overline{\kappa_{\frak p}}(T)]= [F:K(T)]=\deg_Y(P)$. \footnote{geometrically: if $\tilde f_0: \widetilde{\mathcal C}_0\rightarrow \Pp^1_{\kappa_{\frak p}}$ is the special fiber of $\tilde f$, then $\tilde f_0$ is generically \'etale, $\widetilde{\mathcal C}_0$ is geometrically irreducible and $[\overline{\kappa_{\frak p}}(\widetilde{\mathcal C}_0):\overline{\kappa_{\frak p}}(T)]= [F:K(T)]$.}} }
\vskip 3mm

The extension $\varepsilon/\kappa_{\frak p}(T)$ is called the {\it {\rm (}good{\rm )} reduction} of $F/K(T)$ at ${\frak p}$ and  denoted by $F|_{\frak p}/\kappa_{\frak p}(T)$ --- $F|_{u_0}/k(T)$ when ${\frak p}=\langle U-u_0\rangle \subset A=k[U]$. The vertical bar in the notation is meant to distinguish the reduction from the specialization. The extension $F|_{\frak p}/\kappa_{\frak p}(T)$ is $\kappa_{\frak p}$-regular and its branch point set is the reduction, denoted by ${\bf t}_{\frak p}$, of the set ${\bf t}$ modulo an (arbitrary) prime ideal above ${\frak p}$ of 
the integral closure of $A_ {\frak p}$ in $K_ {\frak p}({\bf t})$.

When the residue field $\kappa_{\frak p}$ is algebraically closed, we have this more precise addendum.
Its statement uses the notion of fundamental group representation of an extension $F/\overline K(T)$; it is recalled in \S \ref{ssec:fund-group-basics}.

\vskip 2mm

\noindent
{\bf Addendum to Grothendieck good reduction theorem.}  
\noindent
{\it Under the same assumptions, there is a
specialization isomorphism
\vskip 1,5mm

\centerline{${\rm sp}_{{\frak p}}: \pi_1(\Pp^1\setminus {\mathbf t}, t)_{\overline{K}} \rightarrow \pi_1(\Pp^1\setminus
 {\mathbf t}_{{\frak p}},t)_{\overline{\kappa_{\frak p}}}$}
\vskip 1,5mm

\noindent
which has this further property: if $\phi_{\overline K}: \pi_1(\Pp^1\setminus {\mathbf t}, t)_{\overline{K}} 
\rightarrow G\subset S_d$ is a $\pi_1$-representation of the extension $F\overline K/\overline K(T)$, then 
the morphism
\vskip 1,5mm

\centerline{$\phi_{\overline K} \circ {\rm sp}_{{\frak p}}^{-1}: \pi_1(\Pp^1\setminus
 {\mathbf t}_{{\frak p}},t)_{\overline{\kappa_{\frak p}}} \rightarrow G\subset S_d$} 
 \vskip 1,5mm

\noindent
is a  $\pi_1$-representation of the reduction $F|_{u_0}/\kappa_{\frak p}(T)$.}

\vskip 2mm

Let $P\in A[T,Y]$ be a non-constant polynomial, irreducible in $\overline K[T,Y]$, monic in $Y$,
{\it e.g.} an affine equation of the $K$-regular extension $F/K(T)$.
\vskip 2mm

\noindent
{\bf Polynomial form of the Grothendieck good reduction theorem} {\it Assume that ${\frak p}$ is a good prime 
of $P$ and that assumption {\rm (1)} holds. Then the polynomial ``$P$ modulo ${\frak p}$''  in 
$\kappa_{\frak p}[T,Y]$ is ir\-re\-du\-cible in $\overline{\kappa_{\frak p}}[T,Y]$ and of  
group $G$.}
\vskip 2mm

As explained in \cite{acta2016}, this polynomial conclusion is more precise than the field extension 
conclusion from GRT; the assumption is however also stronger.
Finally we recall the conclusions from \cite{legrand1} about the inertia in specializations.
\vskip 2mm

\noindent
{\bf Specialization Inertia Theorem.} {\it Let $t_0\in \Pp^1(K)\setminus {\bf t}$.
\vskip 0,7mm

\noindent
{\rm (a)} If ${\frak p}$ ramifies in $F_{t_0}/K$, then $F/K(T)$ has vertical ramification at ${\frak p}$ (i.e. condition {\rm (2)}
holds) or $t_0$ meets some branch point modulo ${\frak p}$.
\vskip 1,2mm

\noindent
{\rm (b)} Assume that ${\frak p}$ is a good prime of $F/K(T)$ and assumptions {\rm (1), (4), (5)} holds. 
If for some $i\in \{1,\ldots,r\}$, $t_0$ and $t_i$ meet modulo ${\frak p}$, then the inertia group of 
$F_{t_0}/K$ at ${\frak p}$ is conjugate in $G$ to the cyclic group
\vskip 1mm

\centerline{$\langle g_i^{I_{\frak p}(t_0,t_i)} \rangle$}
\vskip 1mm

\noindent
where $g_i$ is any element of the conjugacy class $C_i$ and $I_{\frak p}(t_0,t_i)$ is the intersection multiplicity of $t_0$ and $t_i$.}

\bibliography{parametric}

\begin{thebibliography}{GMPS15}

\bibitem[Bec91]{Beckmann-specialization}
Sybilla Beckmann.
\newblock On extensions of number fields obtained by specializing branched
  coverings.
\newblock {\em J. Reine Angew. Math.}, 419:27--53, 1991.

\bibitem[BR97]{buhler-reichstein}
J.~Buhler and Z.~Reichstein.
\newblock On the essential dimension of a finite group.
\newblock {\em Compositio Math.}, 106(2):159--179, 1997.

\bibitem[CS82]{cassels-schinzel}
J.~W.~S. Cassels and A.~Schinzel.
\newblock Selmer's conjecture and families of elliptic curves.
\newblock {\em Bull. London Math. Soc.}, 14(4):345--348, 1982.

\bibitem[CT00]{colliot-annals}
Jean-Louis Colliot-Th{\'e}l{\`e}ne.
\newblock Rational connectedness and {G}alois covers of the projective line.
\newblock {\em Ann. of Math. (2)}, 151(1):359--373, 2000.

\bibitem[DD97]{DeDo1}
Pierre D{\`e}bes and Jean-Claude Douai.
\newblock Algebraic covers: field of moduli versus field of definition.
\newblock {\em Annales Sci. E.N.S.}, 30:303--338, 1997.

\bibitem[D{\`e}b99]{DeBB}
Pierre D{\`e}bes.
\newblock Galois covers with prescribed fibers: the {B}eckmann-{B}lack problem.
\newblock {\em Ann. Scuola Norm. Sup. Pisa, {\rm Cl. Sci. (4)}}, 28:273--286,
  1999.

\bibitem[D{\`e}b16]{acta2016}
Pierre D{\`e}bes.
\newblock Reduction and specialization of polynomials.
\newblock {\em Acta Arith.}, 172.2:175--197, 2016.

\bibitem[D{\`e}bar]{IsrJ2}
Pierre D{\`e}bes.
\newblock On the {M}alle conjecture and the self-twisted cover.
\newblock {\em Israel J. Math.}, to appear.

\bibitem[DG12]{DEGha}
Pierre D{\`e}bes and Nour Ghazi.
\newblock Galois covers and the {H}ilbert-{G}runwald property.
\newblock {\em Ann. Inst. Fourier}, 62/3:989--1013, 2012.

\bibitem[DL13]{DeLe1}
Pierre D{\`e}bes and Fran{\c{c}}ois Legrand.
\newblock Specialization results in {G}alois theory.
\newblock {\em Trans. Amer. Math. Soc.}, 365(10):5259--5275, 2013.

\bibitem[FJ04]{FrJa}
Michael~D. Fried and Moshe Jarden.
\newblock {\em Field arithmetic}, volume~11 of {\em Ergebnisse der Mathematik
  und ihrer Grenzgebiete}.
\newblock Springer-Verlag, Berlin, 2004.
\newblock (second edition).

\bibitem[GM71]{GrMu}
Alexandre Grothendieck and Jacob~P. Murre.
\newblock {\em The Tame Fundamental Group of a Formal Neighbourhood of a
  Divisor with Normal Crossings on a Scheme}, volume 208 of {\em LNM}.
\newblock Springer, 1971.

\bibitem[GMPS15]{guest-at-al}
Simon Guest, Joy Morris, Cheryl~E. Praeger, and Pablo Spiga.
\newblock On the maximum orders of elements of finite almost simple groups and
  primitive permutation groups.
\newblock {\em Trans. Amer. Math. Soc.}, 367(11):7665--7694, 2015.

\bibitem[Gro71]{SGA}
Alexandre Grothendieck.
\newblock {\em Rev\^etements \'{e}tales et groupe fondamental}, volume 224 of
  {\em LNM}.
\newblock Springer, 1971.

\bibitem[HM99]{hashimoto-miyake}
Ki-Ichiro Hashimoto and Katsuya Miyake.
\newblock Inverse {G}alois problem for dihedral groups.
\newblock In {\em Number theory and its applications ({K}yoto, 1997)}, volume~2
  of {\em Dev. Math.}, pages 165--181. Kluwer Acad. Publ., Dordrecht, 1999.

\bibitem[JLY02]{JLY}
Christian~U. Jensen, Arne Ledet, and Noriko Yui.
\newblock {\em Generic polynomials. Constructive Aspects of the Inverse Galois
  Problem}.
\newblock Cambridge University Press, 2002.

\bibitem[Leg13]{legrand-thesis}
Fran{\c{c}}ois Legrand.
\newblock Sp\'ecialisations de rev\^etements et th\'eorie inverse de {G}alois.
\newblock {\em {PhD thesis} - Universit\'e Lille 1}, 2013.

\bibitem[Leg15]{legrand2}
Fran{\c{c}}ois Legrand.
\newblock Parametric {G}alois extensions.
\newblock {\em J. Algebra}, 422:187--222, 2015.

\bibitem[Leg16a]{legrand-at-least-1-nonparam}
Fran{\c{c}}ois Legrand.
\newblock On parametric extensions over number fields.
\newblock {\em preprint}, 2016.
\newblock arXiv:1602.06706.

\bibitem[Leg16b]{legrand-hyper-elliptic}
Fran{\c{c}}ois Legrand.
\newblock Twists of super elliptic curves without rational points.
\newblock {\em preprint}, 2016.

\bibitem[Legar]{legrand1}
Fran{\c c}ois Legrand.
\newblock Specialization results and ramification conditions.
\newblock {\em Isr. J. Math.}, (to appear).
\newblock arXiv:1310.2189.

\bibitem[Sch82]{schinzel-book}
Andrzej Schinzel.
\newblock {\em Selected topics on polynomials}.
\newblock University of Michigan Press, Ann Arbor, Mich., 1982.

\bibitem[Sch00]{schinzel-book2000}
A.~Schinzel.
\newblock {\em Polynomials with special regard to reducibility}, volume~77 of
  {\em Encyclopedia of Mathematics and its Applications}.
\newblock Cambridge University Press, Cambridge, 2000.
\newblock With an appendix by Umberto Zannier.

\bibitem[Sel54]{selmer}
Ernst~S. Selmer.
\newblock A conjecture concerning rational points on cubic curves.
\newblock {\em Math. Scand.}, 2:49--54, 1954.

\bibitem[Ser92]{Serre-topics}
Jean-Pierre Serre.
\newblock {\em Topics in Galois Theory}.
\newblock Research Notes in Mathematics. Jones and Bartlett Publishers, 1992.

\bibitem[Tho84]{thompson}
John~G. Thompson.
\newblock Some finite groups which appear as {${\rm Gal}\,L/K$}, where
  {$K\subseteq {\bf Q}(\mu _{n})$}.
\newblock {\em J. Algebra}, 89(2):437--499, 1984.

\end{thebibliography}
\bibliographystyle{alpha}

\end{document}